\documentclass[11pt]{amsart}
\usepackage[utf8]{inputenc}
\usepackage{url}
\usepackage{xspace}
\usepackage{mathscinet}  
\newtheorem{thm}{Theorem}[section]
\newtheorem{prp}[thm]{Proposition}
\newtheorem{lem}[thm]{Lemma}
\newtheorem{cor}[thm]{Corollary}
\theoremstyle{definition}
\newtheorem{defn}[thm]{Definition}
\theoremstyle{remark}
\newtheorem{rem}[thm]{Remark}
\newtheorem{exmpl}[thm]{Example}

\newcommand{\ii}{\mathbf{i}}
\newcommand{\jj}{\mathbf{j}}

\newcommand{\NN}{\mathbb{N}}
\newcommand{\TT}{\mathbb{T}}
\newcommand{\XX}{\mathbb{X}}
\newcommand{\ZZ}{\mathbb{Z}}
\newcommand{\QQ}{\mathbb{Q}}
\newcommand{\QQi}{\mathbb{Q}_\infty}

\newcommand{\cC}{\mathcal{C}}
\newcommand{\cD}{\mathcal{D}}

\newcommand{\cT}{\mathcal{T}}

\newcommand{\Sym}{\mathfrak{S}}

\newcommand{\Sig}{\Sigma}
\newcommand{\sig}{\sigma}

\newcommand{\add}{\operatorname{add}}
\newcommand{\Aut}{\operatorname{Aut}}
\newcommand{\Clt}{{\operatorname{\mathsf{Clt}}}}
\newcommand{\coh}{\operatorname{coh}}
\newcommand{\End}{\operatorname{End}}
\newcommand{\Id}{\operatorname{Id}}
\newcommand{\ind}{\operatorname{ind}}
\newcommand{\Out}{\operatorname{Out}}
\newcommand{\PSL}{\operatorname{PSL}}
\newcommand{\rad}{\operatorname{rad}}
\newcommand{\Ker}{\operatorname{Ker}}
\newcommand{\Coker}{\operatorname{Coker}}
\newcommand{\Groth}{\operatorname{K}_0}
\newcommand{\rk}{\operatorname{rk}}
\newcommand{\Der}{\cD^b}
\newcommand{\lcm}{\operatorname{lcm}}
\newcommand{\Tcan}{T_{\operatorname{can}}}
\newcommand{\md}{\operatorname{mod}}
\newcommand{\slope}{\operatorname{S}}
\newcommand{\slopeset}{\overline{\slope}}
\newcommand{\ibar}[1]{\overline{#1}}

\newcommand{\rightrel}{\mathchoice{\begin{picture}(20,9)\qbezier[7](3,2.6)(9,2.6)(15,2.6)\put(14,2.6){\vector(1,0){3}}\end{picture}}
{\begin{picture}(20,9)\qbezier[7](3,2.6)(9,2.6)(15,2.6)\put(14,2.6){\vector(1,0){3}}\end{picture}}
{\begin{picture}(15,5)\qbezier[6](2,2)(6.5,2)(11,2)\put(10,2){\vector(1,0){3}}\end{picture}}
{\begin{picture}(12,4)\qbezier[7](1.5,1.5)(5,1.5)(8.5,1.5)\put(7.5,1.5){\vector(1,0){3}}\end{picture}}}

\newcommand{\rightar}{\mathchoice{\begin{picture}(20,9)\put(3,2.6){\vector(1,0){14}}\end{picture}}
{\begin{picture}(20,9)\put(3,2.6){\vector(1,0){14}}\end{picture}}
{\begin{picture}(15,5)\put(2,2){\vector(1,0){11}}\end{picture}}
{\begin{picture}(12,4)\put(1.5,1.5){\vector(1,0){9}}\end{picture}}}

\newcommand{\ra}{\rightarrow}

\newcommand{\oq}{\overline{q}}

\newcommand{\ie}{i.e.\@\xspace}

\newcommand{\sD}{\mathsf{D}}
\newcommand{\sE}{\mathsf{E}}

\newcommand{\HVCenter}[1]{\setbox 0=\hbox{#1}%
        \dimen0=\wd0%
        \dimen1=\ht0%
        \divide\dimen0 by 2%
        \divide\dimen1 by 2%
        \hskip -\dimen0%
        \lower \dimen1%
        \box0%
        \hskip -\dimen0}
\newcommand{\HBCenter}[1]{\setbox 0=\hbox{#1}%
        \dimen0=\wd0%
        \dimen1=\ht0%
        \divide\dimen0 by 2%
        \hskip -\dimen0%
        \box0%
        \hskip -\dimen0}
\newcommand{\HTCenter}[1]{\setbox 0=\hbox{#1}%
        \dimen0=\wd0%
        \dimen1=\ht0%
        \divide\dimen0 by 2%
        \hskip -\dimen0%
        \lower \dimen1%
        \box0%
        \hskip -\dimen0}
\newcommand{\RTCenter}[1]{\setbox 0=\hbox{#1}%
        \dimen0=\wd0%
        \dimen1=\ht0%
        \hskip -\dimen0%
        \lower \dimen1%
        \box0%
        \hskip -\dimen0}
\newcommand{\LTCenter}[1]{\setbox 0=\hbox{#1}%
        \dimen0=\wd0%
        \dimen1=\ht0%
        \lower \dimen1%
        \box0%
        \hskip -\dimen0}
\newcommand{\RVCenter}[1]{\setbox 0=\hbox{#1}%
        \dimen0=\wd0%
        \dimen1=\ht0%
        \divide\dimen1 by 2%
        \hskip -\dimen0%
        \lower \dimen1%
        \box0%
        \hskip -\dimen0}
\newcommand{\RBCenter}[1]{\setbox 0=\hbox{#1}%
        \dimen0=\wd0%
        \dimen1=\ht0%
        \hskip -\dimen0%
        \box0%
        \hskip -\dimen0}
\newcommand{\LVCenter}[1]{\setbox 0=\hbox{#1}%
        \dimen1=\ht0%
        \divide\dimen1 by 2%
        \lower \dimen1%
        \box0%
        \hskip -\dimen0}

\begin{document}

\title[Tubular cluster algebras II]{Tubular cluster algebras II: 
Exponential growth}
\date{}

\author{M. Barot}
\address{Instituto de Matemáticas, Universidad Nacional Autónoma de México, 
Ciudad Universitaria, 04510 México, D.F., MEXICO}
\email{barot@matem.unam.mx}

\author{Ch. Geiss}
\address{Instituto de Matemáticas, Universidad Nacional Autónoma de México, 
Ciudad Universitaria, 04510 México, D.F., MEXICO}
\email{christof@matem.unam.mx}

\author{G. Jasso}
\address{Graduate School of Mathematics, Nagoya University, 464-8602 Nagoya,
JAPAN}
\email{jasso.ahuja.gustavo@b.mbox.nagoya-u.ac.jp}

\begin{abstract}
Among the mutation finite cluster algebras the tubular ones are a 
particularly interesting class. 
We show that all tubular (simply laced) cluster algebras are
of exponential growth by two different methods: first by studying 
the automorphism group of the corresponding cluster category
and second by giving explicit sequences of mutations.
\end{abstract} 
\maketitle

\sloppy


\section{Introduction}

Tubular cluster algebras where introduced in~\cite{BarGei12} as a
proper family of cluster algebras, due to their categorification by tubular 
cluster categories. 
These cluster algebras represent three of the 11 exceptional mutation 
finite cluster algebras with skew symmetric exchange matrix~\cite{FeShTu08} 
and one is the surface algebra corresponding to the 4-punctured sphere. 
Figure~\ref{fig:ellipticR} shows representatives of their exchange matrices in 
quiver form. 

\begin{figure}[h]
\begin{center}
  \begin{picture}(300,40)
    \put(10,20){
      \put(-10,15){\RVCenter{\small $\sD_4^{(1,1)}:$}}
      \multiput(0,-15)(30,0){3}{\circle*{3}}
      \multiput(0,15)(30,0){3}{\circle*{3}}
      \multiput(26,-15)(30,30){2}{\vector(-1,0){22}}
      \multiput(4,15)(30,-30){2}{\vector(1,0){22}}
      \multiput(28.5,11)(3,0){2}{\vector(0,-1){22}}
      \multiput(3,-12)(30,0){2}{\vector(1,1){24}}
      \multiput(27,-12)(30,0){2}{\vector(-1,1){24}}
    }
    \put(150,20){
      \put(0,15){\RVCenter{\small $\sE_6^{(1,1)}:$}}
      \multiput(0,0)(30,0){2}{\circle*{3}}
      \multiput(50,-15)(0,30){2}{
      	\multiput(0,0)(30,0){3}{\circle*{3}}
      }
      \multiput(76,15)(30,0){2}{\vector(-1,0){22}}
      \multiput(54,-15)(30,0){2}{\vector(1,0){22}}
      \put(4,0){\vector(1,0){22}}
      \multiput(48.5,11)(3,0){2}{\vector(0,-1){22}}
      \put(46.4,-12.3){\vector(-4,3){12}}
      \put(53,-12){\vector(1,1){24}}
      \put(34.6,2.7){\vector(4,3){12}}
      \put(77,-12){\vector(-1,1){24}}
    }
  \end{picture}
\end{center} 
\begin{center}
  \begin{picture}(300,40)
    \put(0,20){
      \put(0,15){\RVCenter{\small $\sE_7^{(1,1)}$:}}
      \put(0,0){\circle*{3}}
      \multiput(20,-15)(30,0){4}{\circle*{3}}
      \multiput(20,15)(30,0){4}{\circle*{3}}
      \multiput(46,15)(30,0){3}{\vector(-1,0){22}}
      \multiput(24,-15)(30,0){3}{\vector(1,0){22}}
      \multiput(18.5,11)(3,0){2}{\vector(0,-1){22}}
      \put(16.4,-12.3){\vector(-4,3){12}}
      \put(23,-12){\vector(1,1){24}}
      \put(4.6,2.7){\vector(4,3){12}}
      \put(47,-12){\vector(-1,1){24}}
    }
    \put(150,20){
      \put(0,15){\RVCenter{\small $\sE_8^{(1,1)}$:}}
      \put(0,0){\circle*{3}}
      \multiput(20,-15)(30,0){3}{\circle*{3}}
      \multiput(20,15)(30,0){6}{\circle*{3}}
      \multiput(46,15)(30,0){5}{\vector(-1,0){22}}
      \multiput(24,-15)(30,0){2}{\vector(1,0){22}}
      \multiput(18.5,11)(3,0){2}{\vector(0,-1){22}}
      \put(16.4,-12.3){\vector(-4,3){12}}
      \put(23,-12){\vector(1,1){24}}
      \put(4.6,2.7){\vector(4,3){12}}
      \put(47,-12){\vector(-1,1){24}}
    }
  \end{picture}
\end{center} 
\caption{Quivers associated to some elliptic root systems}
\label{fig:ellipticR} 
\end{figure}
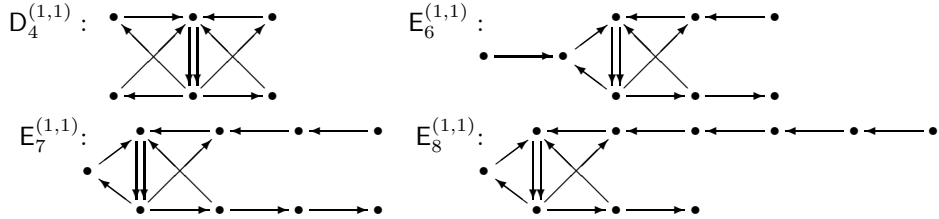

We refer to~\cite{BarGei12} for more details and context on tubular cluster 
algebras.

The tubular cluster algebra of type $(2,2,2,2)$ coincides
with the surface algebra of the 4-punctured sphere. From a mapping class
group argument~\cite[Sec.~11]{FST08} it follows that this algebra is
of exponential growth. In other words, the number of seeds which 
can be obtained from a fixed initial seed by at most $n$ mutations 
is bounded from below by an exponentially growing function of $n$.
Due to the similarity in their categorification one expects this to be true 
also in the remaining three cases which are not related to surfaces.

\begin{thm}
\label{thm:main}
Tubular cluster algebras are of exponential growth.
\end{thm}

We present in this paper two quite different proofs of this result. We think
that both proofs are interesting by themselves as they yield two different 
approaches to this phenomenon.

The first proof, given in Section~\ref{sec:firstproof}, is based on the 
result~\cite[Prop.~7.4]{BaKuLe10} which states that the
group of (isomorphism classes of) triangulated self-equivalences of a
tubular cluster category contains the group $\PSL_2(\ZZ)$. We show that this
descends to an inclusion of  $\PSL_2(\ZZ)$ into the corresponding
cluster modular group. This is, in some sense, an extension of the above 
argument which uses the mapping class group.

The second proof, given in Section~\ref{sec:secondproof}, provides in each of 
the four cases explicit mutation
sequences which directly 
exhibit the exponential growth. This argument is based on a careful
analysis of the lift of mutations (of cluster tilting
objects in the tubular cluster category) to Hübner-mutations 
(of tilting objects in
the corresponding category of coherent sheaves over a weighted projective
line). Another important ingredient in this approach is the close
connection between the exchange graph of Farey triples (a 3-regular tree) and 
the classification of tilting sheaves over a weighted projective line of 
tubular type. 

We would like to mention that Felikson, Shapiro and Tumarkin  recently 
completed their above mentioned classification of mutation finite cluster 
algebras by covering also the skew symmetrizable cases~\cite{FeShTu10}.
Moreover, they determine in~\cite{FeShTu11} (for the orbifold cases) and
in~\cite{FeShThTu12} (with H. Thomas for the remaining exceptional cases) 
the growth rate for all mutation finite cluster algebras. 
In particular, \cite{FeShThTu12}
provides an independent proof for the exponential growth of all tubular 
cluster algebras which is based on a direct study of the corresponding 
cluster modular groups. This includes also the non-simply laced cases.

\subsection*{Acknowledgments}
The second proof is essentially a part of the third author's 
master thesis~\cite{GJ}. 
This research was partially  supported by the grants  
PAPIIT No.~IN117010-2 and  CONACYT No.~81948.

\section{Preliminaries}
\subsection{Exponential growth of graphs} \label{ssec:exp-gr}
In our setting a \emph{graph} $G=(G_0, G_1)$ consists of a vertex set $G_0$ and
a edge set $G_1$ which is a subset of the set of two-element subsets of $G_0$.
A \emph{path in $G$ of length $n$} is a sequence of vertices
$(v_0,v_1,\ldots,v_n)$ such that $\{v_{k-1},v_k\}\in G_1$ for all
$k\in\{1,\ldots,n\}$.  For a vertex $v\in G_0$ we denote by
$v[n]\subset G_0$ the set of vertices which are connected to $v$ by a path
of length less than $n$. Finally, we say that $G$ is of \emph{exponential
growth} if for some $v\in G_0$  we can find a function $f$ of exponential
growth such that $\#(v[n])\geq f(n)$ for all $n\in\NN$. 

For example, we have $\#(v[n])=3(2^n-1)$ for each vertex $v$ of
the $3$-regular tree $\TT_3$.
Thus $\TT_3$ is of exponential growth.

Let $k\in\NN_{\geq 1}$ and $G, H$ be two graphs. By a \emph{$k$-embedding 
of $G$ into $H$} we mean an injective map $i\colon G_0\ra H_0$ such that
for each edge $\{v,w\}\in G_1$ there exists a path in $H$ of length at most
$k$, connecting $i(v)$ and $i(w)$. Obviously, $H$ is of exponential growth
if for some $k$, there exists a $k$-embedding of $\TT_3$ into $H$. 

\subsection{Beginning of the proof} \label{ssec:beg-pf}
By the main result of~\cite{BarGei12}, for a tubular cluster algebra the
exchange graph of seeds is isomorphic to the
exchange graph $\mathcal{G}$ of cluster tilting objects (in the corresponding 
tubular cluster category). Thus, by
the considerations in Section~\ref{ssec:exp-gr}, it is sufficient to construct
a $k$-embedding of a tree $\mathcal{T}$ of exponential growth into the exchange 
graph $\mathcal{G}$.
This will be done in the next two sections by different methods.
In the first proof $\mathcal{T}$ is a rooted binary tree and in the second 
proof it is $\TT_3$.


\section{Cluster modular group and self equivalences}
\label{sec:firstproof}
\subsection{Generalities} \label{ssec:Gp-Gen}
Let $\cC$ be a 2-Calabi-Yau triangulated category, see \cite{Keller08}, 
with split idempotents
over some field.
We denote the suspension functor of $\cC$ by $\Sig$.
We suppose that in $\cC$ there exists a cluster tilting object
$T$ such that there is a cluster structure in the sense
of~\cite{BIRSc09} on the cluster tilting objects reachable from $T$.
Without further mentioning all cluster tilting objects will be assumed to be 
basic.
We fix a cluster tilting object with its decomposition into indecomposable 
direct summands $T=T_1\oplus\cdots\oplus T_n\in\cC$. 

Following Keller~\cite[Sec.~5.5]{Keller11}, we consider the groupoid $\Clt$ 
of \emph{cluster tilting sequences} in $\cC$ reachable from $T$. Its objects
are the sequences $([T'_1],\ldots,[T'_n])$ of isomorphism classes
of indecomposable objects  such that
$T'=\oplus_{k=1}^n T'_k$ is a cluster tilting object in $\cC$ reachable from $T$.
Note that this implies that the summands $T'_i$ are \emph{rigid}, 
\ie $\cC(T'_i,\Sigma T'_i)=0$.
Morphisms
are formal compositions of (per-)mutations of cluster tilting objects, subject 
only to the obvious relations: $\mu_k^2=\Id$ and $\sig\mu_k=\mu_{\sig(k)}\sig$ 
for $k\in\{1,\ldots,n\}$ and each permutation $\sig\in\Sym_n$. 
For convenience we abbreviate $([T_1'],\ldots,[T_n'])=:[T']$.

We say that a triangulated self-equivalence $F$ of $\cC$ is \emph{reachable}
if we have 
$([FT_1],\ldots,[FT_n])\in\Clt$. 
In this case it is not hard to see that $F$ induces a self-equivalence
$\ibar{F}$ of $\Clt$ which we call \emph{induced}.
For a sequence of indices $\ii=(i_s,\ldots,i_2,i_1)$ with 
$i_a\in\{1,\ldots,n\}$ we define
$\mu_{\ii}=\mu_{i_s}\cdots\mu_{i_1}$. For a permutation $\sigma\in\Sym_n$ we set 
$\sigma(\ii)=(\sigma(i_s),\ldots,\sigma(i_1))$.

\begin{prp} \label{Prp_ClustModGp}
Let $F$ and $G$ be two reachable self-equivalences of $\cC$, 
\begin{itemize}
\item[(a)] 
$\ibar{F}=\ibar{G}$
if and only if $([FT_1],\ldots,[FT_n])=([GT_1],\ldots,[GT_n])$.
\item[(b)]
Suppose that for two sequences of indices $\ii$ and $\jj$ and permutations
$\sig,\tau\in\Sym_n$ we have 
$\ibar{F}([T])=\sig\mu_\ii([T_1],\ldots,[T_n])$ and
$\ibar{G}([T])=\tau\mu_\jj([T_1],\ldots,[T_n])$ in $\Clt$. Then
$\ibar{F}\circ\ibar{G}([T])=\tau\sig\mu_{\sig^{-1}(\jj)}\mu_\ii([T])$.
\end{itemize}
\end{prp}

\begin{proof}
(a) Recall that for each $X\in\cC$ there exists a distinguished triangle
$T_X''\ra T_X'\ra X \ra\Sigma T_X''$ with $T_X',T_X''\in\add(T)$. The
\emph{index} $\ind_T(X)$ of $X$ with respect to $T$ is $[T_X']-[T_X'']$
in the split Grothendieck group of $\add(T)$. In~\cite[Sec.~2]{DehKel08}
it is shown that in case $X$ is rigid, $X$ 
is determined up to isomorphism by its index $\ind_T(X)$. 
Moreover, each triangulated self-equivalence of $\cC$ sends
cluster tilting objects to cluster tilting objects. Thus our claim follows
since the cluster tilting objects reachable from $T$ have by hypothesis a 
cluster structure.

(b) Since $\ibar{F}$ is a self-equivalence of $\Clt$ we have
\[
\ibar{F}(\ibar{G}([T]))=\ibar{F}(\tau\mu_\jj([T]))=\tau\mu_\jj(F([T]))=
\tau\mu_\jj\sig\mu_\ii([T])=\tau\sig\mu_{\sig^{-1}(\jj)}\mu_\ii([T]).
\]
\end{proof}

By the above proposition, the induced self-equivalences of $\Clt$ form
a group, which we call the \emph{refined cluster modular group} 
$\Aut_\iota(\Clt)$.

\begin{rem}
The group $\Aut_\iota(\Clt)$ seems to be related to the cluster modular group 
defined by Fock and Goncharov~\cite[1.2.5]{FoGon09a}, see 
also~\cite[p.28]{FoGon06a}. Note that the endomorphism  ring of a cluster
tilting object in $\cC$ is in general not determined by its quiver.
For example, this occurs for the tubular cluster category of weight type
$(2,2,2,2)$, see~\cite[Expl.~6.12]{BeOpWrXX}. This category
can be used to categorify the cluster algebra associated to the sphere
with four punctures~\cite[Rem.~1.2]{BarGei12}.
\end{rem}

\begin{cor} \label{Cor:Aut}
\begin{itemize}
\item[(a)]
We have an injective map
\[
\Aut_\iota(\Clt)\ra\{([T'_1,],\ldots,[T'_n])\in\Clt\mid [T']\equiv [T]\},\  
\ibar{F}\mapsto\ibar{F}([T]),
\]
where $[T']\equiv[T]$ means that the assignment $T_i\mapsto T'_i$ for
$i=1,\ldots,n$ induces an equivalence of categories between $\add(T)$
and $\add(T')$.
\item[(b)]
Suppose, that $\Aut_\iota(\Clt)$ contains a free (non-abelian) subgroup
in two generators. Then 
there is a $k$-embedding of the (rooted) binary tree into the
exchange graph of cluster tilting objects.
\end{itemize}
\end{cor}

\begin{proof} 
Part (a) follows immediately from Proposition~\ref{Prp_ClustModGp}~(a) whereas 
part (b) follows from part~(a) and Proposition~\ref{Prp_ClustModGp}~(b).
\end{proof}

\begin{rem} 
It follows from~\cite{KelYan11}, that in case $\cC$ is the 
generalized cluster category associated to a non-degenerate, Jacobi-finite
quiver with potential, the above map is also surjective.
\end{rem}

\subsection{The tubular case}
Let $\cC$ be a tubular cluster category. Thus $\cC$ is the orbit category
of $\cD:=\cD^b(\coh\XX)$ modulo the self equivalence $\tau^{-1}[1]$, 
for a weighted projective line $\XX$ of tubular type. 
Here we denoted by $\tau$ the Auslander-Reiten automorphism and by $[1]$ the 
shift automorphism of $\cD$.
The canonical projection
$\pi\colon\cD\ra\cC$ is a triangle functor. In this situation, 
$\cC$ fulfills all the requirements of Section~\ref{ssec:Gp-Gen}. Moreover, all
cluster tilting sequences are reachable from any given cluster tilting object,
see \cite[Thm.~8.8]{BaKuLe10}.

It is shown in~\cite[Lemma~6.6]{BaKuLe10} that
all triangulated self-equivalences of $\cD$ are standard. Thus
the isomorphism classes of such self-equivalences form a group which 
can be identified with the derived Picard group $\Aut_s(\cD)$ 
of the corresponding canonical algebra~\cite{RouZim03}. Furthermore, 
it is shown 
in~\cite[Corollary~6.5]{BaKuLe10} that each triangulated 
self-equivalence of $\cC$ can be lifted along $\pi$ to a triangulated 
self-equivalence of 
$\cD$. In particular, the isomorphism classes of triangulated 
self-equivalences $\Aut(\cC)$ of $\cC$ form a factor group
of $\Aut_s(\cD)$.

\begin{prp} \label{Prp:TubIso}
For a tubular cluster category $\cC$ the map
\[
\Aut(\cC)\ra \Aut_\iota(\Clt), F\mapsto \ibar{F}
\]
is an isomorphism of groups.
\end{prp}

\begin{proof} 
By the above observations and the definition, the map $F\mapsto \ibar{F}$ is  
a well-defined, surjective group homomorphism.

Let $F\in\Aut(\cC)$ be such that $\ibar{F}=\Id_\Clt$. As explained above,
there exists a standard self-equivalence $\widetilde{F}$ of $\cD$ which
lifts $F$ along $\pi$. Let $M=M_1\oplus\cdots\oplus M_n$ be a
tilting complex such that $\pi(M_i)=T_i$ for $i=1,\ldots,n$.
Since $\XX$ is tubular, we may assume that $E:=\End_\cD(M)$ is schurian,
\ie $\dim\cD(M_i,M_j)\leq 1$ for all $1\leq i,j\leq n$. By our hypothesis
we may also assume that $\widetilde{F}(M_i)\cong M_i$ for all $i=1,\ldots,n$.
Since $\widetilde{F}$ is standard, it is determined by an element
$\omega\in\Out(E)$ (the group of outer automorphisms of $E$) that fixes
the standard primitive idempotents of $E$, see~\cite[Prop.~2.3]{RouZim03}.
Since $E$ is schurian it follows
that $\omega$ is the identity. Thus $\widetilde{F}$ and $F$ are isomorphic
to the respective identities.
\end{proof}

\subsection{First proof of Theorem~\ref{thm:main}}
By Section~\ref{ssec:beg-pf} and
Corollary~\ref{Cor:Aut}~(b) it is sufficient to 
show that in the tubular case $\Aut_\iota(\Clt)$ contains a free subgroup in two
generators. In fact, by Proposition~\ref{Prp:TubIso} we have 
$\Aut_\iota(\Clt)\cong\Aut(\cC)$ and in~\cite[Prop.~7.4]{BaKuLe10} it is shown,
that $\Aut(\cC)$ is a semidirect product of a finite group by $\PSL_2(\ZZ)$, see
also~\cite[Thm.~6.3]{LenMel00}.\hfill$\Box$


\section{Explicit verification using Farey triples}
\label{sec:secondproof}
\sloppy
\subsection{Hübner mutations}
\label{subsec:Huebner}
In his Ph.D.~thesis~\cite{Hueb}, Hübner investigated tilting objects 
in the category $\coh\XX$ of coherent sheaves over a 
weighted projective line $\XX$.
In this section, we collect  the results from~\cite{Hueb} which are relevant
in our context.

Let $T=\bigoplus_{i=1}^n T_i$ be a tilting sheaf in  
$\coh \XX$. 
Let $Q_T$ be the quiver of the endomorphism algebra of $T$. 
For simplicity, we shall identify 
the vertices of $Q$ with the summands $T_1,\ldots,T_n$.
For each index $i=1,\ldots,n$ we define the morphism
$$
\sigma_i=
\left[\begin{matrix}\sigma_{i1}&\cdots&\sigma_{in} \end{matrix}\right]\colon 
\bigoplus_{h=1}^n  T_{h}^{r_{ih}}\rightarrow T_i
$$ 
where the entries of 
$\sigma_{ih}=\left[\begin{matrix}
\sigma_{ih}^{(1)}&\cdots&\sigma_{ih}^{(r_{ih})}
\end{matrix}\right]$ constitute a set of morphisms 
$\sigma_{ih}^{(a)}\colon T_h\rightarrow T_i$ which are mapped to a basis under 
the canonical projection
$$
\rad(T_h,T_i)\rightarrow\rad(T_h,T_i)/\rad^2(T_h,T_i).
$$
Note that $r_{ih}$ is the number of arrows $h\rightarrow i$ in $Q_T$.
Similarly we define a morphism 
$$
\rho_i=\left[\begin{matrix}\rho_{i1}\\\vdots\\ \rho_{in}\end{matrix}\right]
\colon T_i\rightarrow \bigoplus_{h=1}^n T_h^{r_{hi}},
$$
where  $\rho_{ih}=\sigma_{ih}^\top$.

\begin{prp}\cite[Prop.~2.6, 2.8]{Hueb}
\label{Prop:Huebner2.6-2.8}
With the previous notation, we have the following results.
\begin{itemize}
\item[(a)]
For each index $i$ the morphism
$\sigma_i$ (resp.~$\rho_i$) is either a monomorphism or an epimorphism
in $\coh\XX$.
\item[(b)]
For each index $i$ the morphism
$\sigma_i$ is mono (resp.~epi) if and only if $\rho_i$ is a mono (resp.~epi).
\item[(c)] Let $T_k^\ast=\Ker\sigma_k\oplus\Coker\rho_k$. Then 
$T_k^\ast\oplus\bigoplus_{j\neq k}T_j$ is again a tilting object in $\coh \XX$.
\end{itemize}
\end{prp} 

In view of Proposition~\ref{Prop:Huebner2.6-2.8}(b), exactly one of 
$\Ker\sigma_k$ and $\Coker\rho_k$ is 
non-zero. This 
allows us to separate the vertices of $Q_T$ into two classes: 
$T_k$ is called a \emph{Hübner-source} (resp.~a \emph{Hübner-sink}) 
in case $\sigma_k$ and $\rho_k$ are mono (resp.~epi).

\begin{rem}
\label{rem:warning}
The following warning seems in place, see
also \cite[Bem.~3.3]{Hueb}: it is possible to have 
two tilting sheaves $T=\bigoplus_{i=1}^n T_i$ and 
$T'=\bigoplus_{i=1}^n T'_i$ with isomorphic quivers $Q_T$ and $Q_{T'}$, 
such that some vertex $T_i$ is a Hübner-source of $Q_i$ whereas its 
corresponding vertex $T'_i$ is a Hübner-sink of $Q_{T'}$. 
\end{rem}

\begin{defn}
Let $T=\bigoplus_{i=1}^n T_i$ be a tilting sheaf in $\coh\XX$. Then for each 
index $k\in\{1,\ldots,n\}$ we define the \emph{mutation of $T$ in direction $k$}
to be the tilting sheaf
$$
\mu_k(T)=T_k^\ast\oplus\bigoplus_{j\neq k} T_j,
$$
where $T_k^\ast=\Ker\sigma_k\oplus\Coker\rho_k$.
\end{defn}

\begin{rem}
In \cite{Hueb}
the tilting sheaf $\mu_k(T)$ is called \emph{reflection} at the source 
or sink due to the similarity with the Bernstein-Gelfand-Ponomarev 
reflections~\cite{BeGePo}.
However, given the role they play as lifts of mutations in the context  of 
cluster algebras we prefer this new terminology.
\end{rem}

In view of Remark~\ref{rem:warning}, 
conditions which characterize Hübner-sinks 
and Hübner-sources are important.

\begin{prp}\cite[Bem.~2.10, Kor.~3.5]{Hueb}
\label{prop:Huebner-char}
A sink (resp.~source) of a quiver is always a Hübner-sink (resp.~Hübner-source).
A successor of a Hübner-sink is again a Hübner-sink and a predecessor of a 
Hübner-source is again a Hübner-source. Furthermore if the endomorphism 
algebra is given by its quiver and some relations, then any relation starts in 
a Hübner-source and ends in a Hübner-sink.
\end{prp}

The preceding conditions are not sufficient to decide always whether a given 
vertex is a Hübner-source or a Hübner-sink. To give a sufficient 
characterization we need some notions.

We recall that we denote by  $\cD=\cD^b(\coh \XX)$ the bounded derived category
of $\coh\XX$. 
Recall from \cite{GeLe87} that 
there are two $\ZZ$-linear forms, $\rk$ and $\deg$ on 
$\Groth(\coh\XX)=\Groth(\Der(\coh\XX))$,
called the \emph{rank} and \emph{degree}.
Furthermore the \emph{slope} is defined as $\slope=\frac{\deg}{\rk}$.
Note that each tilting sheaf $T$ in $\coh \XX$ gives rise to a triangulated 
equivalence between the bounded derived categories $\Der(\coh \XX)$ and 
$\Der(\md A)$ where $A=\End(T)$. Since such an equivalence induces an 
isomorphism between the corresponding Grothendieck groups,
we can evaluate rank and degree on (classes of)  $A$-modules.

\begin{prp}\cite[Bem.~3.3]{Hueb}
\label{prop:h-source-sink}
Let $T=\bigoplus_{i=1}^n T_i$ be a tilting sheaf in $\coh \XX$. Further denote 
by $S_i$ the simple right $\End(T)$-module associated to $T_i$.
Then $T_i$ is a Hübner-source if and only if $\rk(S_i)>0$ or $\rk(S_i)=0$ and 
$\deg(S_i)>0$. 
\end{prp}

\begin{exmpl}\label{exmpl:can}
Let $(p_1,p_2,\ldots,p_t)$ be the weight sequence of $\XX$ and 
$p=\lcm(p_1,p_2,\ldots,p_t)$.
The \emph{canonical configuration} $\Tcan$, see \cite{GeLe87}, is a tilting 
sheaf whose endomorphism algebra is a canonical algebra in the sense of 
Ringel~\cite[Sec.~3.7]{Ri84}. The following picture shows its quiver. 
There are $t-2$ relations from the unique source to the unique sink of the 
quiver.
\begin{center}
\begin{picture}(265,100)
\put(5,50){%
\put(48,0){%
	\put(15,40){\vector(1,0){15}}
	\put(15,10){\vector(1,0){15}}
	\put(15,-40){\vector(1,0){15}}%
	}%
\put(106,40){%
	\put(6,0){\line(1,0){9}}
	\multiput(20.5,0)(3,0){3}{\line(1,0){1}}
	\put(33,0){\vector(1,0){11}}
}%
\put(106,10){%
	\put(6,0){\line(1,0){9}}
	\multiput(20.5,0)(3,0){3}{\line(1,0){1}}
	\put(33,0){\vector(1,0){11}}
}%
\put(106,-40){%
	\put(6,0){\line(1,0){9}}
	\multiput(20.5,0)(3,0){3}{\line(1,0){1}}
	\put(33,0){\vector(1,0){11}}
}%
\put(4.242640687,4.242640687){\vector(1,1){31.51471863}}
\put(5.820855001,1.45521375){\vector(4,1){28.35829}}
\put(4.242640687,-4.242640687){\vector(1,-1){31.51471863}}
\put(54,0){%
	 \put(154.2426407,35.75735931){\vector(1,-1){31.51471863}}
	 \put(155.820855,8.54478625){\vector(4,-1){28.35829}}
	 \put(154.2426407,-35.75735931){\vector(1,1){31.51471863}}
	}%
\put(-1,0){\HVCenter{$0$}}
\put(48,0){%
	\put(0,40){\HVCenter{$\frac{p}{p_1}$}}
	\put(0,10){\HVCenter{$\frac{p}{p_2}$}}
	\put(0,-40){\HVCenter{$\frac{p}{p_t}$}}
	}%
\put(95,0){%
	\put(0,40){\HVCenter{$\frac{2p}{p_1}$}}
	\put(0,10){\HVCenter{$\frac{2p}{p_2}$}}
	\put(0,-40){\HVCenter{$\frac{2p}{p_t}$}}
	}%
\put(180,0){%
	\put(0,40){\HVCenter{$\frac{(p_1-1)p}{p_1}$}}
	\put(0,10){\HVCenter{$\frac{(p_2-1)p}{p_2}$}}
	\put(0,-40){\HVCenter{$\frac{(p_t-1)p}{p_t}$}}
	}%
\put(48,-15){\HVCenter{$\vdots$}}%
\put(95,-15){\HVCenter{$\vdots$}}%
\put(180,-15){\HVCenter{$\vdots$}}%
\put(250,0){\HVCenter{\small $p$}}%
}%
\end{picture}
\end{center}
The indecomposable direct summands of $\Tcan$ 
have rank $1$ and degree $j\frac{p}{p_i}$ as shown in the picture above. 
\end{exmpl}

The next result, though interesting in its own, permits to calculate the rank 
in concrete examples.

\begin{prp}\cite[Thm.~4.6]{Hueb}
\label{prop:rank-additivity}
For each tilting sheaf $T$ the rank function is an additive function on the 
quiver $Q_T$ with relations. More precisely, if $T=\bigoplus_{i=1}^n T_i$ then 
for each indecomposable direct summand $T_i$ we have
$$
2\rk(T_i)=\sum_{j\rightar i} \rk(T_j)+\sum_{i\rightar j} \rk(T_j)
-\sum_{j\rightrel i} \rk(T_j) -\sum_{i\rightrel j} \rk(T_j),
$$
where the summation has to be taken over all arrows and relations ending 
in $i$. 
\end{prp} 

The thesis of Hübner \cite{Hueb} contains also a precise description of 
the effect on the endomorphism algebra of tilting sheaves under mutation, in 
terms of arrows and relations.
Let $T=\bigoplus_{i=1}^n T_i$ be a tilting sheaf for $\coh \XX$ and let 
$\mu_k(T)=\bigoplus_{i\neq k} T_i\oplus T'_k$ be the mutation in direction $k$. 
We state here only the version for the Hübner-source --- the one for a 
Hübner-sink is completely dual and therefore left to the interested reader.

\begin{prp}\cite[Kor.~4.16]{Hueb}
If $T_i$ is a Hübner-source, then the quiver with relations $Q'$ for 
$\mu_k(T)$ is obtained from $Q$ as follows. 
\begin{itemize}
\item[(i)]
The quiver $Q'$ has the same vertices as $Q$.
\item[(ii)]
For each pair of arrows $i\rightar k\rightar j$ an arrow $i\rightar j$ is added.
\item[(iii)]
For each pair of an arrow $k\rightar i$ and a relation $k\rightrel j$ a 
relation $i\rightrel j$ is added.
\item[(iv)]
Each arrow $i\rightar k$ is replaced by a relation $i\rightrel k$.
\item[(v)]
Each arrow $k\rightar i$ is replaced by an arrow $i\rightar k$.
\item[(vi)]
Each relation $k\rightrel i$ is replaced by an arrow $k\rightar i$.
\item[(vii)]
Pairs of parallel relations and arrows are successively canceled.
\item[(viii)]
All remaining arrows and relations remain unchanged.
\end{itemize}
\end{prp}

\begin{rem}
\label{rem:rel}
If each relation $i\rightrel j$ is replaced by an arrow $j\rightar i$ then the 
corresponding mutation rule is precisely the mutation of diagrams as formulated
in \cite{FZ03}. Further we note that this definition is compatible with the 
mutation of $\ZZ_2$-graded quivers, introduced by Amiot and Oppermann 
in~\cite[Def. 6.2]{AmOp}.
\end{rem}

\begin{exmpl}
\label{exmpl:2222}
We consider the weight sequence $(2,2,2,2)$. Let $\Tcan$ be the canonical 
configuration, see Example~\ref{exmpl:can}. We label the vertices of the 
quiver of $\End(\Tcan)$ by $1,2,\ldots,6$ such that $1$ is the source and $6$ 
is the sink. We consider the mutation sequence $\mu_6\mu_3\mu_2$, 
and indicate the slope $\frac{\deg(T_i)}{\rk(T_i)}$ in the quivers:
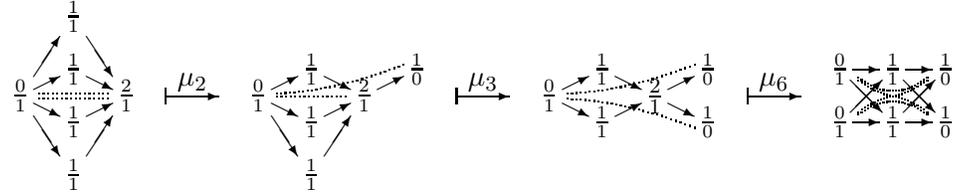
\begin{center}
\begin{picture}(350,70)
\put(0,35){
\put(0,0){
	\put(0,2){
		\put(0,0){\HVCenter{\small$\frac{0}{1}$}}
		\multiput(20,-30)(0,20){4}{\HVCenter{\small$\frac{1}{1}$}}
		\put(40,0){\HVCenter{\small$\frac{2}{1}$}}
	}
	\put(5,2.5){\vector(2,1){10}}
	\put(5,-2.5){\vector(2,-1){10}}
	\put(5,7.5){\vector(2,3){10}}
	\put(5,-7.5){\vector(2,-3){10}}
	\put(25,7.5){\vector(2,-1){10}}
	\put(25,-7.5){\vector(2,1){10}}
	\put(25,22.5){\vector(2,-3){10}}
	\put(25,-22.5){\vector(2,3){10}}
	
	\multiput(0,-1)(0,2){2}{
		\qbezier[15](7,0)(20,0)(33,0)
	}
}
\put(65,0){
	\put(-10,0){\vector(1,0){20}}
	\put(-10,-3){\line(0,1){6}}
	\put(0,4){\HBCenter{$\mu_2$}}
}
\put(90,0){
	\put(0,2){
		\put(0,0){\HVCenter{\small$\frac{0}{1}$}}
		\multiput(20,-30)(0,20){3}{\HVCenter{\small$\frac{1}{1}$}}
		\put(40,0){\HVCenter{\small$\frac{2}{1}$}}
		\put(60,10){\HVCenter{\small$\frac{1}{0}$}}
	}
	\put(45,2.5){\vector(2,1){10}}
	\put(5,2.5){\vector(2,1){10}}
	\put(5,-2.5){\vector(2,-1){10}}
	\put(5,-7.5){\vector(2,-3){10}}
	\put(25,7.5){\vector(2,-1){10}}
	\put(25,-7.5){\vector(2,1){10}}
	\put(25,-22.5){\vector(2,3){10}}
	\qbezier[15](7,0)(20,0)(33,0)
	\qbezier[28](7,1)(30,2)(55,12)
}
\put(175,0){
	\put(-10,0){\vector(1,0){20}}
	\put(-10,-3){\line(0,1){6}}
	\put(0,4){\HBCenter{$\mu_3$}}
}
\put(200,0){
	\put(0,2){
		\put(0,0){\HVCenter{\small$\frac{0}{1}$}}
		\multiput(20,-10)(0,20){2}{\HVCenter{\small$\frac{1}{1}$}}
		\put(40,0){\HVCenter{$\frac{2}{1}$}}
		\multiput(60,-10)(0,20){2}{\HVCenter{\small$\frac{1}{0}$}}
	}
	\put(45,2.5){\vector(2,1){10}}
	\put(45,-2.5){\vector(2,-1){10}}
	\put(5,2.5){\vector(2,1){10}}
	\put(5,-2.5){\vector(2,-1){10}}
	\put(25,7.5){\vector(2,-1){10}}
	\put(25,-7.5){\vector(2,1){10}}
	\qbezier[28](7,1)(30,2)(55,12)
	\qbezier[28](7,-1)(30,-2)(55,-12)
}
\put(285,0){
	\put(-10,0){\vector(1,0){20}}
	\put(-10,-3){\line(0,1){6}}
	\put(0,4){\HBCenter{$\mu_6$}}
}
\put(310,0){
	\put(0,2){
		\multiput(0,-10)(0,20){2}{\HVCenter{\small$\frac{0}{1}$}}
		\multiput(20,-10)(0,20){2}{\HVCenter{\small$\frac{1}{1}$}}
		\multiput(40,-10)(0,20){2}{\HVCenter{\small$\frac{1}{0}$}}
	}
	\multiput(0,0)(20,0){2}{
		\put(4,6){\vector(1,-1){12}}
		\put(4,-6){\vector(1,1){12}}
		\multiput(5,-10)(0,20){2}{\vector(1,0){10}}
	}
	\qbezier[22](7,7)(20,-3)(33,7)
	\qbezier[11](7,6)(14,1)(20,0)\qbezier[11](33,-6)(26,-1)(20,0)
	\qbezier[22](7,-7)(20,3)(33,-7)
	\qbezier[11](7,-6)(14,-1)(20,0)\qbezier[11](33,6)(26,1)(20,0)
}
}
\end{picture}
\end{center}   
\end{exmpl}
  
\subsection{Farey triples}
We resume some basic properties of Farey triples, 
see also~\cite[Sec.~2]{Na11}.
First, we extend $\QQ$ to $\QQi=\QQ\cup\{\infty\}$ and observe that each 
element $q\in\QQ$ defines uniquely two integers $d(q)$ and $r(q)$ which are 
relatively prime and such that
$$
q=\frac{d(q)}{r(q)},\quad r(q)> 0.
$$
Furthermore, we define $d(\infty)=1$ and $r(\infty)=0$.

\begin{defn}
For a pair $p,q\in \QQi$ the \emph{Farey distance} is defined as
$$
\Delta(p,q)=\left| d(p)r(q)-d(q)r(p)\right|.
$$
If $\Delta(p,q)=1$, then $p,q$ are called \emph{Farey neighbours}.
A triple $\{q_1,q_2,q_3\}$ of elements of $\QQi$ which are pairwise Farey 
neighbours is called a \emph{Farey triple}.

Given $p,q\in\QQi$ the \emph{Farey sum} $\oplus$ and \emph{Farey difference} 
$\ominus$ are defined by
$$
p\oplus q=\frac{d(p)+d(q)}{r(p)+r(q)},\quad
p\ominus q=\frac{d(p)-d(q)}{r(p)-r(q)}.
$$
If $\oq=\{p,q,r\}$ is a Farey triple then the \emph{mutation $\mu_p(\oq)$ in 
direction} $p$ is defined by
$$
\mu_p(\oq)=\begin{cases}
\{q\ominus r,q,r\},&\quad\text{if $q<p<r$ or $r<p<q$,}\\
\{q\oplus r,q,r\},&\quad \text{if $p<\min(q,r)$ or $p>\max(q,r)$}.
\end{cases}
$$

\end{defn}

\begin{lem}
\label{lem:farey_mutation}
To any two Farey neighbours there exist exactly two Farey triples containing 
them.
If $\oq$ is a Farey triple then $\mu_p(\oq)$ is again a Farey triple for any 
$p\in\oq$. Moreover, the mutation of Farey triples is involutive, in the sense 
that $\mu_{p'}\mu_p(\oq)=\oq$ if $\oq=\{p,q,r\}$ and $\mu_p(\oq)=\{p',q,r\}$.
\end{lem}

\begin{proof}
Let $p=\frac{a}{b}$ and $q=\frac{c}{d}$ be Farey neighbours. We may assume that
$p>q$, that is $ad-bc=1$. Now suppose that 
$\{\frac{a}{b},\frac{c}{d},\frac{e}{f}\}$ is a Farey triple. Then
\begin{align}
\label{eq:farey1}
af-be&=\varepsilon\in\{1,-1\}\\
\label{eq:farey2}
cf-de&=\varphi\in\{1,-1\}
\end{align}
By multiplying \eqref{eq:farey2} by $a$ and using \eqref{eq:farey1} we get
$e=\varepsilon a-\varphi c$. Similarly we get $f=\varepsilon b-\varphi d$.
 
Therefore, the four possible choices for the signs 
$\varepsilon, \varphi\in\{1,-1\}$ lead to precisely two possible solutions for 
$\frac{e}{f}$. This proves the first statement and the rest of the proof of the
following result is straightforward and left to the interested reader.
\end{proof}

\begin{rem}
Similarly as in \cite{BarGei12} we define the \emph{complexity} $c(q)$ of 
$q\in\QQi$ to be $\left| d(q)\right|+r(q)+\left|d(q)-r(q)\right|$. It follows  
easily that $\{1,0,\infty\}$ is the unique Farey-triple of minimal sum of the 
complexities, and that each other Farey triple can be mutated in a 
unique direction 
so that its sum of complexities decreases. 
Consequently, the exchange graph of the Farey triples form a $3$-regular tree 
under  mutations.
\end{rem}

\begin{lem}
\label{lem:Farey_triple_eq}
If $\{p,q,r\}$ is a Farey triple with 
$p<q<r$ then
$$
2q\ominus p=q\oplus r,\quad
2q\ominus r=q\oplus p,\quad
2p\ominus q=2r\ominus q=p\ominus r,
$$
where $2\frac{a}{b}\ominus\frac{c}{d}=\frac{2a-c}{2b-d}$.
\end{lem}

\begin{proof}
Write $p=\frac{a}{b}$, $q=\frac{c}{d}$ and $r=\frac{e}{f}$. The statements 
follow easily by applying the equations
$cb-ad=1$, $eb-af=1$ and $ed-cf=1$. 
\end{proof}

\subsection{Tilting sheaves realizing a Farey triple} \label{ssec:TiFa}
Given a tilting sheaf $T=\bigoplus_{i=1}^n T_i$ we define its \emph{slope set} 
to be $\slopeset(T)=\{\slope(T_i)\mid i=1,\ldots,n\}$.
We say that a tilting sheaf $T$ \emph{realizes} a Farey triple $\oq$
if $\slopeset(T)=\oq$.  In Example~\ref{exmpl:2222} we have given a tilting 
sheaf realizing the Farey triple $\{0,1,\infty\}$ for the weight sequence 
$(2,2,2,2)$.
We now consider the remaining weight sequences $(3,3,3)$, $(4,4,2)$ and 
$(6,3,2)$.
In each case we indicate how a tilting sheaf $T$ realizing the Farey triple 
$\{0,1,\infty\}$ may be obtained from the canonical configuration $\Tcan$ by a 
sequence of mutations.
For this, we label the vertices of the quiver of $\End(\Tcan)$ by $1,\ldots,n$ 
in such a way that 
$1$ is the source, $n$ is the sink, and the remaining vertices are labeled 
$2,\ldots,n-1$ from left to right and top to bottom in the picture of the 
quiver given in Example~\ref{exmpl:can}.

For the (tubular) weight sequence $(3,3,3)$ 
we use the sequence $7,5,3,8,1$ of mutations (\ie the composition of 
mutations $\mu_1\mu_8\mu_3\mu_5\mu_7$) to obtain from $\Tcan$ a
tilting sheaf $T$ realizing $\{0,1,\infty\}$.  
The quiver of $\End(T)$ is isomorphic to the first quiver in
Figure~\ref{fig:ex-real}.
The following table contains the information about degree and rank of 
the indecomposable direct summands of $T$.
\begin{center}
\begin{picture}(190,40)
\multiput(0,0)(0,20){3}{\line(1,0){190}}
\put(0,0){\line(0,1){40}}
\multiput(30,0)(20,0){9}{\line(0,1){40}}
\put(15,28){\HBCenter{$i$}}
\put(15,8){\HBCenter{$\frac{\deg(T_i)}{\rk(T_i)}$}}
\put(40,28){\HBCenter{$u_1$}}
\put(60,28){\HBCenter{$u_2$}}
\put(80,28){\HBCenter{$u_3$}}
\multiput(40,8)(20,0){3}{\HBCenter{$\frac{1}{1}$}}
\put(100,28){\HBCenter{$v_1$}}
\put(120,28){\HBCenter{$v_2$}}
\put(100,8){\HBCenter{$\frac{0}{1}$}}
\put(120,8){\HBCenter{$\frac{0}{2}$}}
\put(140,28){\HBCenter{$w_1$}}
\put(160,28){\HBCenter{$w_2$}}
\put(180,28){\HBCenter{$w_3$}}
\multiput(140,8)(20,0){3}{\HBCenter{$\frac{1}{0}$}}
\end{picture}
\end{center}
\begin{figure}[!ht]
\begin{picture}(330,100)
\put(24,40){
	\put(0,-2){
		\put(0,0){\HBCenter{\small$v_2$}}
		\put(0,36){\HBCenter{\small$u_1$}}
		\put(-27,18){\HBCenter{\small$w_3$}}
		\put(-27,-18){\HBCenter{\small$u_2$}}
		\put(27,18){\HBCenter{\small$w_2$}}
		\put(27,-18){\HBCenter{\small$u_3$}}
		\put(0,-36){\HBCenter{\small$w_1$}}
		\put(42,0){\HBCenter{\small$v_1$}}
	}
	\put(0,9){\vector(0,1){18}}
	\multiput(5.4,-3.6)(-27,-18){2}{\vector(3,-2){16.2}}
	\multiput(-5.4,-3.6)(27,-18){2}{\vector(-3,-2){16.2}}
	\put(-5.4,32.4){\vector(-3,-2){16.2}}
	\put(5.4,32.4){\vector(3,-2){16.2}}
	\qbezier[13](5.4,3.6)(13.5,9)(21.6,14.4)
	\qbezier[13](-5.4,3.6)(-13.5,9)(-21.6,14.4)
	\qbezier[13](0,-10)(0,-18)(0,-26)
	\put(27,-9){\vector(0,1){18}}
	\put(-27,-9){\vector(0,1){18}}
	\put(6,0){\vector(1,0){30}}
}
\put(145,4){
	\put(0,2){
		\put(-36,0){\HVCenter{\small$u_1$}}
		\put(0,0){\HVCenter{\small$v_3$}}
		\put(36,0){\HVCenter{\small$w_1$}}
		\put(-18,27){\HVCenter{\small$u_2$}}
		\put(18,27){\HVCenter{\small$w_2$}}
		\put(0,54){\HVCenter{\small$v_2$}}
		\put(-36,54){\HVCenter{\small$w_3$}}
		\put(36,54){\HVCenter{\small$u_3$}}
		\put(0,86){\HVCenter{\small$v_1$}}
	}
	\put(-18,27){\put(-3.6,-5.4){\vector(-2,-3){10.8}}}
	\multiput(-18,27)(18,27){2}{\put(3.6,-5.4){\vector(2,-3){10.8}}}
	\multiput(-18,27)(18,-27){2}{\put(3.6,5.4){\vector(2,3){10.8}}}
	\put(18,27){\put(3.6,-5.4){\vector(2,-3){10.8}}}
	\multiput(0,54)(36,0){2}{\put(-6,0){\vector(-1,0){24}}}
	\put(0,63){\vector(0,1){14}}
	\qbezier[14](-11,27)(0,27)(11,27)
	\put(18,27){\qbezier[13](3.6,5.4)(9,13.5)(14.4,21.6)}
	\put(-18,27){\qbezier[13](-3.6,5.4)(-9,13.5)(-14.4,21.6)}
}
\put(280,40){
	\put(0,0){
		\put(0,-18){\HVCenter{\small $v_3$}}
		\put(0,18){\HVCenter{\small $u_3$}}
		\put(-27,0){\HVCenter{\small $w_2$}}
		\put(-27,-36){\HVCenter{\small $u_1$}}
		\put(-27,36){\HVCenter{\small $v_2$}}
		\put(-57,0){\HVCenter{\small $w_1$}}
		\put(27,0){\HVCenter{\small $w_4$}}
		\put(27,-36){\HVCenter{\small $u_2$}}
		\put(27,36){\HVCenter{\small $v_1$}}
		\put(57,0){\HVCenter{\small $w_3$}}
	}
	\multiput(-27,36)(54,0){2}{\put(0,-8){\vector(0,-1){20}}}
	\put(0,18){\put(0,-8){\vector(0,-1){20}}}
	\multiput(0,-18)(0,36){2}{
		\put(5.4,3.6){\vector(3,2){16.2}}
	}
	\multiput(0,-18)(0,36){2}{
		\put(-5.4,3.6){\vector(-3,2){16.2}}
	}
	\put(27,-36){\put(-5.4,3.6){\vector(-3,2){16.2}}}
	\put(-27,-36){\put(5.4,3.6){\vector(3,2){16.2}}}
	\put(-33,0){\vector(-1,0){18}}
	\put(33,0){\vector(1,0){18}}
	\put(-27,0){\qbezier[13](6,4)(13.5,9)(21,14)}
	\put(27,0){\qbezier[13](-6,4)(-13.5,9)(-21,14)}
	\multiput(-27,-18)(54,0){2}{\qbezier[13](0,10)(0,0)(0,-10)}
}
\end{picture}
\caption{Examples of quivers of $\End(T)$ for $T$ realizing a Farey triple. 
The weight sequences from left to right are $(3,3,3)$, $(4,4,2)$ and $(6,3,2)$.}
\label{fig:ex-real}
\end{figure}
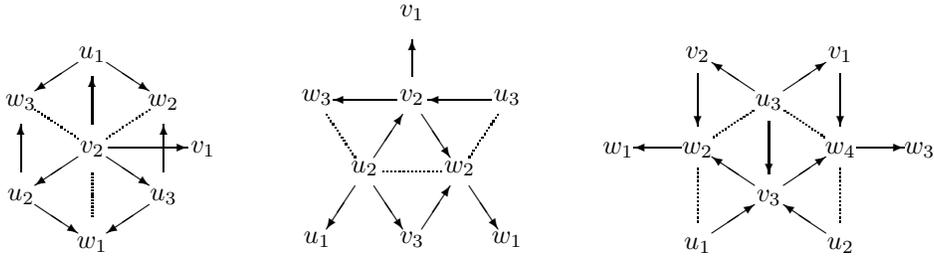
For the (tubular) weight sequence $(4,4,2)$ we use 
the mutation sequence $3,6,9,4,7,9,7,8,1,3$ to 
obtain from $\Tcan$ a tilting sheaf $T$ realizing $\{0,1,\infty\}$. 
The quiver of $\End(T)$ is isomorphic to the second quiver
in Figure~\ref{fig:ex-real}. The following table contains 
the information about degree and rank of the indecomposable direct
summands of $T$. 
\begin{center}
\begin{picture}(210,40)
\multiput(0,0)(0,20){3}{\line(1,0){210}}
\put(0,0){\line(0,1){40}}
\multiput(30,0)(20,0){10}{\line(0,1){40}}
\put(15,28){\HBCenter{$i$}}
\put(15,8){\HBCenter{$\frac{\deg(T_i)}{\rk(T_i)}$}}
\put(40,28){\HBCenter{$u_1$}}
\put(60,28){\HBCenter{$u_2$}}
\put(80,28){\HBCenter{$u_3$}}
\multiput(40,8)(40,0){2}{\HBCenter{$\frac{0}{1}$}}
\put(60,8){\HBCenter{$\frac{0}{2}$}}
\put(100,28){\HBCenter{$v_1$}}
\put(120,28){\HBCenter{$v_2$}}
\put(140,28){\HBCenter{$v_3$}}
\multiput(100,8)(40,0){2}{\HBCenter{$\frac{1}{1}$}}
\put(120,8){\HBCenter{$\frac{2}{2}$}}
\put(160,28){\HBCenter{$w_1$}}
\put(180,28){\HBCenter{$w_2$}}
\put(200,28){\HBCenter{$w_3$}}
\multiput(160,8)(40,0){2}{\HBCenter{$\frac{1}{0}$}}
\put(180,8){\HBCenter{$\frac{2}{0}$}}
\end{picture}
\end{center}

For the (tubular) weight sequence $(6,3,2)$ we use the
mutation sequence $5,10,8,6,10,3,4,8,9,7,2,4,6,5$ to 
obtain from $\Tcan$ a tilting sheaf $T$ realizing $\{0,1,\infty\}$. 
The quiver of $\End(T)$ is isomorphic to the third quiver 
in Figure~\ref{fig:ex-real}. 
The following table contains the information about degree and rank of the 
indecomposable direct summands of $T$.
\begin{center}
\begin{picture}(230,40)
\multiput(0,0)(0,20){3}{\line(1,0){230}}
\put(0,0){\line(0,1){40}}
\multiput(30,0)(20,0){11}{\line(0,1){40}}
\put(15,28){\HBCenter{$i$}}
\put(15,8){\HBCenter{$\frac{\deg(T_i)}{\rk(T_i)}$}}
\put(40,28){\HBCenter{$u_1$}}
\put(60,28){\HBCenter{$u_2$}}
\put(80,28){\HBCenter{$u_3$}}
\multiput(40,8)(20,0){2}{\HBCenter{$\frac{0}{1}$}}
\put(80,8){\HBCenter{$\frac{0}{2}$}}
\put(100,28){\HBCenter{$v_1$}}
\put(120,28){\HBCenter{$v_2$}}
\put(140,28){\HBCenter{$v_3$}}
\put(160,28){\HBCenter{$v_4$}}
\multiput(100,8)(40,0){2}{\HBCenter{$\frac{1}{1}$}}
\multiput(120,8)(40,0){2}{\HBCenter{$\frac{2}{2}$}}
\put(180,28){\HBCenter{$w_1$}}
\put(200,28){\HBCenter{$w_2$}}
\put(220,28){\HBCenter{$w_3$}}
\multiput(180,8)(20,0){2}{\HBCenter{$\frac{1}{0}$}}
\put(220,8){\HBCenter{$\frac{2}{0}$}}
\end{picture}
\end{center}

\subsection{Explicit mutation sequences}
For each of the tubular weight sequences $(2,2,2,2)$, $(3,3,3)$, $(4,4,2)$ 
and $(6,3,2)$ we give explicit mutation sequences which transform a given 
tilting sheaf realizing a Farey triple into another one realizing a mutated 
Farey triple. 

We start by considering the case where the weight sequence of $\XX$ is 
$(2,2,2,2)$.
Let $\cT_{2,2,2,2}$ be the set of (isomorphism classes of) tilting sheaves 
\[
T=\bigoplus_{i\in I} T_i \text{ with } I = \{u_1,u_2,v_1,v_2,w_1,w_2\}, 
\]
such that 
\begin{itemize}
\item 
$\slope(T_{x_1})=\slope(T_{x_2})$ for all $x\in\{u,v,w\}$,
\item
$\slopeset(T)$ is a Farey triple, 
\item
the quiver $Q$ with relations of $\End(T)$ looks as follows:
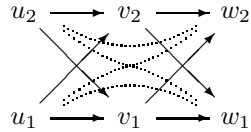
\begin{center}
\begin{picture}(100,50)
\put(5,3){
	\put(0,0){\HBCenter{\small $u_1$}}
	\put(0,40){\HBCenter{\small $u_2$}}
	\put(40,0){\HBCenter{\small $v_1$}}
	\put(40,40){\HBCenter{\small $v_2$}}
	\put(80,0){\HBCenter{\small $w_1$}}
	\put(80,40){\HBCenter{\small $w_2$}}
}
\put(5,5){
	\multiput(6,6)(40,0){2}{\vector(1,1){26}}
	\multiput(10,0)(40,0){2}{\vector(1,0){20}}
	\multiput(6,34)(40,0){2}{\vector(1,-1){26}}
	\multiput(10,40)(40,0){2}{\vector(1,0){20}}
	\qbezier[30](15,5)(40,20)(65,5)
	\qbezier[30](15,35)(40,20)(65,35)
	\qbezier[15](15,7)(30,18)(40,20)\qbezier[15](65,33)(50,22)(40,20)
	\qbezier[15](15,33)(30,22)(40,20)\qbezier[15](65,7)(50,18)(40,20)
}
\end{picture}
\end{center}
\end{itemize}
Note that $\cT_{2,2,2,2}$ is not empty as shown in 
Section~\ref{ssec:TiFa}.
Define the mutation sequences 
$\mu_u=\mu_{u_1}\mu_{u_2}$, $\mu_v=\mu_{v_1}\mu_{v_2}$
and $\mu_w=\mu_{w_1}\mu_{w_2}$.
 
\begin{prp} \label{prp2222}
Let $\XX$ be a weighted projective line with weight sequence $(2,2,2,2)$.
Then for any $T\in\cT_{2,2,2,2}$ and $x\in\{u,v,w\}$ 
we have that $\mu_x(T)\in \cT$ and 
$\slopeset(\mu_x(T))=\mu_{q_x}(\slopeset(T))$, where $q_x=\slope(T_{x_1})$.
\end{prp}

\begin{proof}
Let $\slopeset(T)=\{q_u,q_v,q_w\}$ with $q_u<q_v<q_w$. Since 
$\slope(T_{u_i})<\slope(T_{v_i})<\slope(T_{w_i})$ for each $i=1,2$ it follows 
that 
$\slope(T_{u_i})=q_u$, $\slope(T_{v_i})=q_v$ and 
$\slope(T_{w_i})=q_w$ for $i=1,2$.

If $x=u$ (resp.~$x=w$) then the mutation takes place in two Hübner-sources 
(resp.~Hübner-sinks). It follows from Proposition~\ref{prop:rank-additivity} 
and the symmetry of the quiver that $\rk(T_{u_1})=\rk(T_{u_2})$ and similarly 
$\rk(T_{v_1})=\rk(T_{v_2})$ and $\rk(T_{w_1})=\rk(T_{w_2})$. 
Therefore the corresponding equalities hold also for the degree. 
We denote $r_u=\rk(T_{u_i})$, $d_f=\rk(T_{u_i})$ and similarly define 
$r_v$, $d_v$, $r_w$ and $d_w$.

Now assume that $x=u$. By mutation in $u_2$ we obtain the following quiver. 
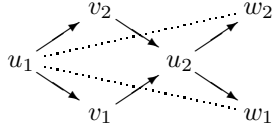
\begin{center}
\begin{picture}(100,50)
\put(5,3){
	\put(0,20){\HBCenter{\small $u_1$}}
	\put(60,20){\HBCenter{\small $u_2$}}
	\put(30,0){\HBCenter{\small $v_1$}}
	\put(30,40){\HBCenter{\small $v_2$}}
	\put(90,0){\HBCenter{\small $w_1$}}
	\put(90,40){\HBCenter{\small $w_2$}}
}
\put(5,5){
	\multiput(6,24)(60,0){2}{\vector(3,2){16}}
	\put(36,36){\vector(3,-2){16}}
	\multiput(6,16)(60,0){2}{\vector(3,-2){16}}
	\put(36,4){\vector(3,2){16}}
	\qbezier[30](9,22)(45,30)(81,38)
	\qbezier[30](9,18)(45,10)(81,2)
}
\end{picture}
\end{center}
Thus by Proposition~\ref{prop:Huebner-char} the vertex $T_{u_1}$ is a Hübner 
source of the quiver of $\mu_{u_2}(T)$. Hence the quiver of $\mu_u(T)$ is 
isomorphic to $Q$.  Using that the rank and the degree are additive on exact 
sequences we get easily that $\rk(T'_{u_i})=2 r_v-r_u$ and
$\deg(T'_{u_i})=2d_v-d_u$ for $i=1,2$, where $T'_{u_1}$, and $T'_{u_2}$ denote the
two summands of $\mu_u(T)$ which are obtained instead of $T_{u_1}$ and $T_{u_2}$ 
by mutation.
Using Lemma~\ref{lem:Farey_triple_eq} we get
$$
\slope(T'_{u_i})=\frac{2 d_v-d_u}{2 r_v-r_u}=\frac{d_v+d_w}{r_v+r_w}=
q_v\oplus q_w.
$$
This shows that $\slopeset(\mu_f(T))=\mu_{q_u}(\slopeset(T))$ and hence the 
result in case $x=u$. 

The case where $x=w$ is completely similar and the case $x=v$ is also similar 
with the unique difference that it is possible for the vertices 
$T_{v_1}$, $T_{v_2}$ to be a Hübner-source or a Hübner-sink (both of the same 
kind).
\end{proof}

We now focus on the case where the weight sequence is $(3,3,3)$.
Here we look at three possible quivers with relations.

\begin{figure}[!ht]
\begin{picture}(326,80)
\put(20,0){
\put(24,40){
	\put(-47,34){\HBCenter{$Q_u$:}}
	\put(0,-2){
		\put(0,0){\HBCenter{\small$v_2$}}
		\put(0,36){\HBCenter{\small$u_1$}}
		\put(-27,18){\HBCenter{\small$w_3$}}
		\put(-27,-18){\HBCenter{\small$u_2$}}
		\put(27,18){\HBCenter{\small$w_2$}}
		\put(27,-18){\HBCenter{\small$u_3$}}
		\put(0,-36){\HBCenter{\small$w_1$}}
		\put(42,0){\HBCenter{\small$v_1$}}
	}
	\put(0,-27){\vector(0,1){18}}
 	\put(-27,18){\multiput(6,-4)(0,-36){2}{\vector(3,-2){15}}}
	\put(27,18){\multiput(-6,-4)(0,-36){2}{\vector(-3,-2){15}}}
	\put(-6,32){\vector(-3,-2){15}}
	\put(6,32){\vector(3,-2){15}}
	\qbezier[13](6,-4)(13.5,-9)(21,-14)
	\qbezier[13](-6,-4)(-13.5,-9)(-21,-14)
	\qbezier[13](0,10)(0,18)(0,26)
	\put(27,-9){\vector(0,1){18}}
	\put(-27,-9){\vector(0,1){18}}
	\put(6,0){\vector(1,0){30}}
}
\put(144,40){
	\put(-47,34){\HBCenter{$Q_v$:}}
	\put(0,-2){
		\put(0,0){\HBCenter{\small$v_2$}}
		\put(0,36){\HBCenter{\small$u_1$}}
		\put(-27,18){\HBCenter{\small$w_3$}}
		\put(-27,-18){\HBCenter{\small$u_2$}}
		\put(27,18){\HBCenter{\small$w_2$}}
		\put(27,-18){\HBCenter{\small$u_3$}}
		\put(0,-36){\HBCenter{\small$w_1$}}
		\put(42,0){\HBCenter{\small$v_1$}}
	}
	\put(0,9){\vector(0,1){18}}
	\multiput(6,-4)(-27,-18){2}{\vector(3,-2){15}}
	\multiput(-6,-4)(27,-18){2}{\vector(-3,-2){15}}
	\put(-6,32){\vector(-3,-2){15}}
	\put(6,32){\vector(3,-2){15}}
	\qbezier[13](6,4)(13.5,9)(21,14)
	\qbezier[13](-6,4)(-13.5,9)(-21,14)
	\qbezier[13](0,-10)(0,-18)(0,-26)
	\put(27,-9){\vector(0,1){18}}
	\put(-27,-9){\vector(0,1){18}}
	\put(6,0){\vector(1,0){30}}
}
\put(264,40){
	\put(-47,34){\HBCenter{$Q_w$:}}
	\put(0,-2){
		\put(0,0){\HBCenter{\small$v_2$}}
		\put(0,36){\HBCenter{\small$u_1$}}
		\put(-27,18){\HBCenter{\small$w_3$}}
		\put(-27,-18){\HBCenter{\small$u_2$}}
		\put(27,18){\HBCenter{\small$w_2$}}
		\put(27,-18){\HBCenter{\small$u_3$}}
		\put(0,-36){\HBCenter{\small$w_1$}}
		\put(42,0){\HBCenter{\small$v_1$}}
	}
	\multiput(0,9)(0,-36){2}{\vector(0,1){18}}
	\put(-27,18){\put(6,-4){\vector(3,-2){15}}}
	\put(27,18){\put(-6,-4){\vector(-3,-2){15}}}
	\put(6,-4){\vector(3,-2){15}}
	\put(-6,-4){\vector(-3,-2){15}}
	\multiput(0,0)(-27,-54){2}{\qbezier[13](5.4,32.4)(13.5,27)(21.6,21.6)}
	\multiput(0,0)(27,-54){2}{\qbezier[13](-5.4,32.4)(-13.5,27)(-21.6,21.6)}
	\multiput(-27,18)(54,0){2}{\qbezier[13](0,-10)(0,-18)(0,-26)}
	\put(6,0){\vector(1,0){30}}
}
}
\end{picture}
\caption{Quivers with relations for the weight sequence $(3,3,3)$.}
\label{fig:quivers-333}
\end{figure}
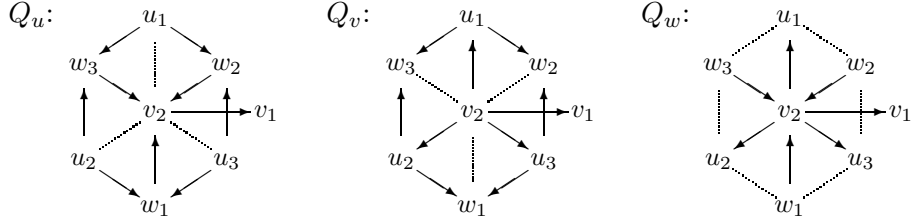

Let $\cT_{(3,3,3)}$ be the set of (isomorphism classes of) tilting sheaves 
\[
T=\oplus_{i\in I} T_i \text{ with } I=\{u_1,u_2,u_3,v_1,v_2,w_1,w_2,w_3\},
\]
such that  
\begin{itemize}
\item the rank and degree of the indecomposable summands of $T$ are 
related as shown in the following table:
\vspace{.5ex}
\begin{center}
\begin{picture}(190,40)
\multiput(0,0)(0,20){3}{\line(1,0){190}}
\put(0,0){\line(0,1){40}}
\multiput(30,0)(20,0){9}{\line(0,1){40}}
\put(15,28){\HBCenter{$i$}}
\put(15,8){\HBCenter{$\frac{\deg(T_i)}{\rk(T_i)}$}}
\put(40,28){\HBCenter{$u_1$}}
\put(60,28){\HBCenter{$u_2$}}
\put(80,28){\HBCenter{$u_3$}}
\multiput(40,8)(20,0){3}{\HBCenter{$\frac{a}{b}$}}
\put(100,28){\HBCenter{$v_1$}}
\put(120,28){\HBCenter{$v_2$}}
\put(100,8){\HBCenter{$\frac{c}{d}$}}
\put(120,8){\HBCenter{$\frac{2c}{2d}$}}
\put(140,28){\HBCenter{$w_1$}}
\put(160,28){\HBCenter{$w_2$}}
\put(180,28){\HBCenter{$w_3$}}
\multiput(140,8)(20,0){3}{\HBCenter{$\frac{e}{f}$}}
\end{picture}
\end{center}
\item
$\slopeset(T)=\{\frac{a}{b},\frac{c}{d}, \frac{e}{f}\}$ is a Farey triple,
\item
the quiver with relations of $\End(T)$ is one of the three quivers 
of Figure~\ref{fig:quivers-333}. 
\end{itemize}
Note that $\cT_{3,3,3}$ is not empty by the construction in 
Section~\ref{ssec:TiFa}.

We define the mutation sequences
\begin{align*}
\mu_u&=\mu_{u_3}\mu_{u_2}\mu_{u_1},\\ 
\mu_v&=\mu_{u_3}\mu_{v_2}\mu_{v_1}\mu_{u_2}\mu_{w_3}
\mu_{u_1}\mu_{v_2},\\ 
\mu_w&=\mu_{w_1}\mu_{w_2}\mu_{w_3}.
\end{align*}

\begin{prp} \label{prp333}
Let $\XX$ be a weighted projective line with weight sequence $(3,3,3)$.
Then for any $T\in\cT_{(3,3,3)}$ and $x\in\{u,v,w\}$ we have 
$\mu_x(T)\in \cT_{(3,3,3)}$ 
and $\slopeset(\mu_x(T))=\mu_{q_x}(\slopeset(T))$, 
where $q_x=\slope(T_{x_1})$.
\end{prp}

\begin{proof}
Let $T\in\cT_{(3,3,3)}$ and denote $T'=\mu_x(T)$. 
We start with the case $x=u$. 
Let us first assume that the quiver of $\End(T)$ is $Q_c$.
Then $\frac{c}{d}<\frac{a}{b}<\frac{e}{f}$ and therefore 
$\frac{a}{b}=\frac{c}{d}\oplus\frac{e}{f}$. Using 
Proposition~\ref{prop:rank-additivity} we get that $b=d+f$ and then $a=c+e$.
We will distinguish whether $T_{u_1}$ is a Hübner-source or Hübner-sink using 
Proposition~\ref{prop:h-source-sink} and observe that if $b=2d$ then it follows
that $a>2c$. Hence we get the two cases:
\begin{itemize}
\item[(a)]
$T_{f_i}$ is a Hübner-source, which is equivalent to $b\geq 2d$, or
\item[(b)]
$T_{f_i}$ is a Hübner-sink, which is equivalent to $b< 2d$.
\end{itemize}
In case (a), the effect of the mutation sequence $\mu_u$ on the quiver of 
$\End(T)$ is as depicted in the next illustration. 
We observe that by Proposition~\ref{prop:h-source-sink} also $T_{u_2}$ is a 
Hübner-source in $\mu_{u_1}(T)$ and similarly $T_{u_3}$ is a Hübner-source of 
$\mu_{u_2}\mu_{u_1}(T)$.
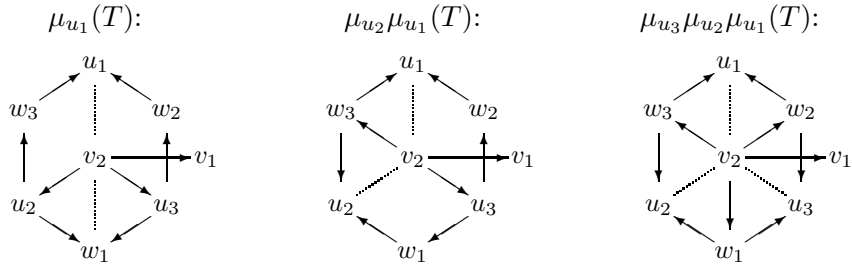
\begin{center}
\begin{picture}(326,100)
\put(20,0){
\put(24,40){
	\put(0,50){\HBCenter{$\mu_{u_1}(T)$:}}
	\put(0,-2){
		\put(0,0){\HBCenter{\small$v_2$}}
		\put(0,36){\HBCenter{\small$u_1$}}
		\put(-27,18){\HBCenter{\small$w_3$}}
		\put(-27,-18){\HBCenter{\small$u_2$}}
		\put(27,18){\HBCenter{\small$w_2$}}
		\put(27,-18){\HBCenter{\small$u_3$}}
		\put(0,-36){\HBCenter{\small$w_1$}}
		\put(42,0){\HBCenter{\small$v_1$}}
	}
	\multiput(0,-36)(0,36){2}{\qbezier[13](0,9)(0,18)(0,27)}
 	\put(-27,-18){\put(6,-4){\vector(3,-2){15}}}
	\put(27,-18){\put(-6,-4){\vector(-3,-2){15}}}
	\put(-27,18){\put(6,4){\vector(3,2){15}}}
	\put(27,18){\put(-6,4){\vector(-3,2){15}}}
	\put(-6,-4){\vector(-3,-2){15}}
	\put(6,-4){\vector(3,-2){15}}
	\put(27,-9){\vector(0,1){18}}
	\put(-27,-9){\vector(0,1){18}}
	\put(6,0){\vector(1,0){30}}
}
\put(144,40){
	\put(0,50){\HBCenter{$\mu_{u_2}\mu_{u_1}(T)$:}}
	\put(0,-2){
		\put(0,0){\HBCenter{\small$v_2$}}
		\put(0,36){\HBCenter{\small$u_1$}}
		\put(-27,18){\HBCenter{\small$w_3$}}
		\put(-27,-18){\HBCenter{\small$u_2$}}
		\put(27,18){\HBCenter{\small$w_2$}}
		\put(27,-18){\HBCenter{\small$u_3$}}
		\put(0,-36){\HBCenter{\small$w_1$}}
		\put(42,0){\HBCenter{\small$v_1$}}
	}
	\qbezier[13](0,9)(0,18)(0,27)
 	\put(0,-36){\put(-6,4){\vector(-3,2){15}}}
	\put(27,-18){\put(-6,-4){\vector(-3,-2){15}}}
	\put(-27,18){\put(6,4){\vector(3,2){15}}}
	\put(27,18){\put(-6,4){\vector(-3,2){15}}}
	\put(-6,4){\vector(-3,2){15}}
	\qbezier[13](-6,-4)(-13.5,-9)(-21,-14)
	\put(6,-4){\vector(3,-2){15}}
	\put(27,-9){\vector(0,1){18}}
	\put(-27,9){\vector(0,-1){18}}
	\put(6,0){\vector(1,0){30}}
}
\put(264,40){
	\put(0,50){\HBCenter{$\mu_{u_3}\mu_{u_2}\mu_{u_1}(T)$:}}
	\put(0,-2){
		\put(0,0){\HBCenter{\small$v_2$}}
		\put(0,36){\HBCenter{\small$u_1$}}
		\put(-27,18){\HBCenter{\small$w_3$}}
		\put(-27,-18){\HBCenter{\small$u_2$}}
		\put(27,18){\HBCenter{\small$w_2$}}
		\put(27,-18){\HBCenter{\small$u_3$}}
		\put(0,-36){\HBCenter{\small$w_1$}}
		\put(42,0){\HBCenter{\small$v_1$}}
	}
	\qbezier[13](0,9)(0,18)(0,27)
 	\put(0,-36){\put(-6,4){\vector(-3,2){15}}}
	\put(0,-36){\put(6,4){\vector(3,2){15}}}
	\put(-27,18){\put(6,4){\vector(3,2){15}}}
	\put(27,18){\put(-6,4){\vector(-3,2){15}}}
	\put(-5.4,3.6){\vector(-3,2){15}}
	\put(5.4,3.6){\vector(3,2){15}}
	\qbezier[13](-6,-4)(-13.5,-9)(-21,-14)
	\qbezier[13](6,-4)(13.5,-9)(21,-14)
	\put(27,9){\vector(0,-1){18}}
	\put(-27,9){\vector(0,-1){18}}
	\put(0,-9){\vector(0,-1){18}}
	\put(6,0){\vector(1,0){30}}
}
}
\end{picture}
\end{center}
The new summands $T'_{u_i}$ are obtained as cokernels: 
$$
T_{u_i}\rightarrow T_{w_j}\oplus T_{w_h}\rightarrow T'_{u_i}
$$
where $j,h$ are such that $\{h,i,j\}=\{1,2,3\}$. Consequently,
for each $i=1,2,3$  the rank and the degree of $T'_{u_i}$ equals $2f-b=f-d$ and
$2e-a=e-c$ respectively. Hence $\mu(T'_{u_i})=\frac{e}{f}\ominus\frac{c}{d}$. 
This shows that $\slopeset(\mu_u(T))=\mu_{\frac{a}{b}}(\slopeset(T))$.

In case (b) we obtain the following sequence where we now observe that 
$T_{u_2}$ is a Hübner-sink of $\mu_{u_1}(T)$ and $T_{u_3}$ is a Hübner-sink of 
$\mu_{u_2}\mu_{u_1}(T)$.
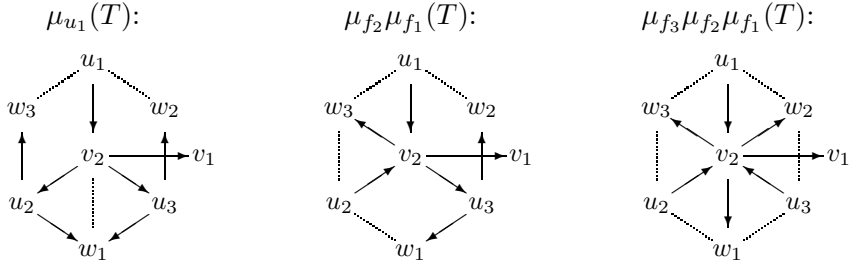
\begin{center}
\begin{picture}(326,100)
\put(20,0){
\put(24,40){
	\put(0,50){\HBCenter{$\mu_{u_1}(T)$:}}
	\put(0,-2){
		\put(0,0){\HBCenter{\small$v_2$}}
		\put(0,36){\HBCenter{\small$u_1$}}
		\put(-27,18){\HBCenter{\small$w_3$}}
		\put(-27,-18){\HBCenter{\small$u_2$}}
		\put(27,18){\HBCenter{\small$w_2$}}
		\put(27,-18){\HBCenter{\small$u_3$}}
		\put(0,-36){\HBCenter{\small$w_1$}}
		\put(42,0){\HBCenter{\small$v_1$}}
	}
	\qbezier[13](0,-9)(0,-18)(0,-27)
 	\put(0,27){\vector(0,-1){18}}
	\put(27,-18){\put(-6,-4){\vector(-3,-2){15}}}
	\put(-27,-18){\put(6,-4){\vector(3,-2){15}}}
	\put(-6,-4){\vector(-3,-2){15}}
	\put(6,-4){\vector(3,-2){15}}
	\put(27,-9){\vector(0,1){18}}
	\put(-27,-9){\vector(0,1){18}}
	\put(6,0){\vector(1,0){30}}
	\put(0,36){\qbezier[13](-6,-4)(-13.5,-9)(-21,-14)}
	\put(0,36){\qbezier[13](6,-4)(13.5,-9)(21,-14)}
}
\put(144,40){
	\put(0,50){\HBCenter{$\mu_{f_2}\mu_{f_1}(T)$:}}
	\put(0,-2){
		\put(0,0){\HBCenter{\small$v_2$}}
		\put(0,36){\HBCenter{\small$u_1$}}
		\put(-27,18){\HBCenter{\small$w_3$}}
		\put(-27,-18){\HBCenter{\small$u_2$}}
		\put(27,18){\HBCenter{\small$w_2$}}
		\put(27,-18){\HBCenter{\small$u_3$}}
		\put(0,-36){\HBCenter{\small$w_1$}}
		\put(42,0){\HBCenter{\small$v_1$}}
	}
	\put(-27,-18){\qbezier[13](0,9)(0,18)(0,27)}
 	\put(0,27){\vector(0,-1){18}}
	\put(27,-18){\put(-6,-4){\vector(-3,-2){15}}}
	\put(-27,-18){\put(6,4){\vector(3,2){15}}}
	\put(6,-4){\vector(3,-2){15}}
	\put(-6,4){\vector(-3,2){15}}
	\put(27,-9){\vector(0,1){18}}
	\put(6,0){\vector(1,0){30}}
	\put(0,36){\qbezier[13](-6,-4)(-13.5,-9)(-21,-14)}
	\put(0,36){\qbezier[13](6,-4)(13.5,-9)(21,-14)}
	\put(0,-36){\qbezier[13](-6,4)(-13.5,9)(-21,14)}
}
\put(264,40){
	\put(0,50){\HBCenter{$\mu_{f_3}\mu_{f_2}\mu_{f_1}(T)$:}}
	\put(0,-2){
		\put(0,0){\HBCenter{\small$v_2$}}
		\put(0,36){\HBCenter{\small$u_1$}}
		\put(-27,18){\HBCenter{\small$w_3$}}
		\put(-27,-18){\HBCenter{\small$u_2$}}
		\put(27,18){\HBCenter{\small$w_2$}}
		\put(27,-18){\HBCenter{\small$u_3$}}
		\put(0,-36){\HBCenter{\small$w_1$}}
		\put(42,0){\HBCenter{\small$v_1$}}
	}
	\multiput(-27,-18)(54,0){2}{\qbezier[13](0,9)(0,18)(0,27)}
 	\put(0,27){\vector(0,-1){18}}
	\put(0,-9){\vector(0,-1){18}}
	\put(27,-18){\put(-6,4){\vector(-3,2){15}}}
	\put(-27,-18){\put(6,4){\vector(3,2){15}}}
	\put(5.4,3.6){\vector(3,2){16.2}}
	\put(-5.4,3.6){\vector(-3,2){16.2}}
	\put(6,0){\vector(1,0){30}}
	\put(0,36){\qbezier[13](-6,-4)(-13.5,-9)(-21,-14)}
	\put(0,36){\qbezier[13](6,-4)(13.5,-9)(21,-14)}
	\put(0,-36){\qbezier[13](-6,4)(-13.5,9)(-21,14)}
	\put(0,-36){\qbezier[13](6,4)(13.5,9)(21,14)}
}
}
\end{picture}
\end{center}
For each $i=1,2,3$ the rank and the degree of $T'_{u_i}$ equal 
$2d-b=d-f$ and $2c-a=c-e$ respectively. Hence 
$\mu(T'_{u_i})=\frac{c}{d}\ominus\frac{e}{f}$ and again 
$\slopeset(\mu_u(T))=\mu_{\frac{a}{b}}(\slopeset(T))$ follows.

Still having $x=u$ we now look at the case where $\End(T)$ has as quiver $Q_u$ 
(resp. $Q_w$). Then it is easily observed that $\mu_u$ yields the inverse 
process as in case (a) (resp. in case (b)) above.   
This concludes the assertion if $x=u$. The case $x=w$ is handled completely 
similar. We now assume that $x=v$. Since the arguments are the same as used for
$x=u$ except that the mutation sequence is substantially longer we will 
use certain abbreviations. 

We start considering first the case where $\End(T)$ has quiver $Q_w$.
Then $\frac{e}{f}<\frac{c}{d}<\frac{a}{b}$ and therefore we conclude as before 
that $c=a+e$ and $d=b+f$. 
We shall distinguish the following three cases:
$$
\text{(a)\ \ }f\leq b,\quad\quad
\text{(b)\ \ } \tfrac{1}{2}f\leq b<f,\quad\quad
\text{(c)\ \ }
\tfrac{1}{2}f<b.
$$

In case (a) the summand $T_{v_2}$ is a Hübner-source and $T_{u_1}$ a Hübner 
source of $\mu_{v_2}(T)$. The first two steps in the mutation sequence 
look as follows. 
We have indicated the rank of each summand as superscript. The degree is 
obtained by replacing $b$ by $a$ and $f$ by $e$.

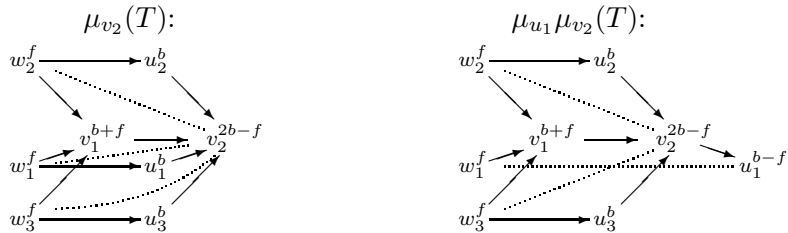
\begin{center}
\begin{picture}(280,90)
\put(0,0){
	\put(0,30){
		\put(40,43){\HBCenter{$\mu_{v_2}(T)$:}}
		\put(0,-2){
			\put(0,30){\HBCenter{\scriptsize $w_2^{f}$}}
			\put(0,-10){\HBCenter{\scriptsize $w_1^{f}$}}
			\put(0,-30){\HBCenter{\scriptsize $w_3^{f}$}}
			\put(30,0){\HBCenter{\scriptsize $v_1^{b+f}$}}
			\put(50,30){\HBCenter{\scriptsize $u_2^{b}$}}
			\put(50,-10){\HBCenter{\scriptsize $u_1^{b}$}}
			\put(50,-30){\HBCenter{\scriptsize $u_3^{b}$}}
			\put(80,0){\HBCenter{\scriptsize$v_2^{2b-f}$}}
		}
		\qbezier[30](12,26)(40,15)(68,4)
		\qbezier[33](12,-26)(50,-25)(72,-6)
		\qbezier[27](12,-8.8)(40,-5)(62,-2)
		\multiput(0,30)(50,0){2}{\put(6,-6){\vector(1,-1){16}}}
		\multiput(0,-30)(50,0){2}{\put(6,6){\vector(1,1){18}}}
		\multiput(0,-10)(50,0){2}{\put(6,2){\vector(3,1){13}}}
		\put(6,30){\vector(1,0){38}}
		\put(6,-10){\vector(1,0){38}}
		\put(6,-30){\vector(1,0){38}}
		\put(42,0){\vector(1,0){20}}
	}
	\put(170,30){
		\put(40,43){\HBCenter{$\mu_{u_1}\mu_{v_2}(T)$:}}
		\put(0,-2){
			\put(0,30){\HBCenter{\scriptsize $w_2^{f}$}}
			\put(0,-10){\HBCenter{\scriptsize $w_1^{f}$}}
			\put(0,-30){\HBCenter{\scriptsize $w_3^{f}$}}
			\put(30,0){\HBCenter{\scriptsize $v_1^{b+f}$}}
			\put(50,30){\HBCenter{\scriptsize $u_2^{b}$}}
			\put(110,-10){\HBCenter{\scriptsize $u_1^{b-f}$}}
			\put(50,-30){\HBCenter{\scriptsize $u_3^{b}$}}
			\put(80,0){\HBCenter{\scriptsize$v_2^{2b-f}$}}
		}
		\qbezier[30](12,26)(40,15)(68,4)
		\qbezier[30](12,-26)(40,-15)(68,-4)
		\qbezier[50](12,-10)(55,-10)(98,-10)
		\multiput(0,30)(50,0){2}{\put(6,-6){\vector(1,-1){16}}}
		\multiput(0,-30)(50,0){2}{\put(6,6){\vector(1,1){18}}}
		\put(0,-10){\put(6,2){\vector(3,1){13}}}
		\put(80,0){\put(6,-2){\vector(3,-1){13}}}
		\put(6,30){\vector(1,0){38}}
		\put(6,-30){\vector(1,0){38}}
		\put(42,0){\vector(1,0){20}}
	}
}
\end{picture}
\end{center}
Now $T_{w_3}$ is a Hübner-source of $\mu_{u_1}\mu_{v_2}(T)$ since a relation 
starts in $w_3$. The resulting quiver after mutation is shown in the next 
illustration on the left.
Then  $T_{u_2}$ is a Hübner-source of $\mu_{w_3}\mu_{u_1}\mu_{v_2}(T)$
since $f\geq b$.
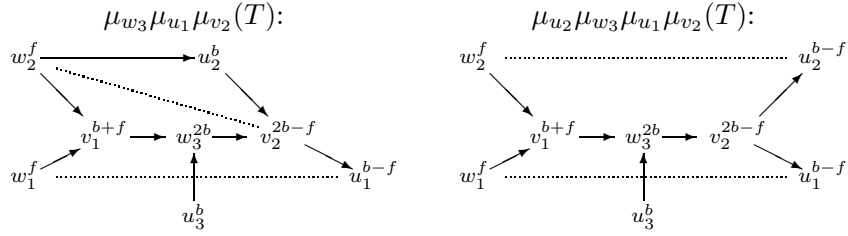
\begin{center}
\begin{picture}(280,95)
\put(0,0){
	\put(0,30){
		\put(64,43){\HBCenter{$\mu_{w_3}\mu_{u_1}\mu_{v_2}(T)$:}}
		\put(0,-2){
			\put(0,30){\HBCenter{\scriptsize $w_2^{f}$}}
			\put(0,-15){\HBCenter{\scriptsize $w_1^{f}$}}
			\put(64,0){\HBCenter{\scriptsize $w_3^{2b}$}}
			\put(30,0){\HBCenter{\scriptsize $v_1^{b+f}$}}
			\put(70,30){\HBCenter{\scriptsize $u_2^{b}$}}
			\put(132,-15){\HBCenter{\scriptsize $u_1^{b-f}$}}
			\put(64,-30){\HBCenter{\scriptsize $u_3^{b}$}}
			\put(100,0){\HBCenter{\scriptsize$v_2^{2b-f}$}}
		}
		\qbezier[45](12,26)(50,15)(88,4)
		\qbezier[60](12,-15)(65,-15)(118,-15)
		\multiput(0,30)(70,0){2}{\put(6,-6){\vector(1,-1){16}}}
		\put(0,-15){\put(6,3){\vector(2,1){15}}}
		\put(100,0){\put(6,-3){\vector(2,-1){18}}}
		\put(6,30){\vector(1,0){58}}
		\put(64,-24){\vector(0,1){18}}
		\put(40,0){\vector(1,0){13}}
		\put(71,0){\vector(1,0){13}}
	}
	\put(170,30){
		\put(64,43){\HBCenter{$\mu_{u_2}\mu_{w_3}\mu_{u_1}\mu_{v_2}(T)$:}}
		\put(0,-2){
			\put(0,30){\HBCenter{\scriptsize $w_2^{f}$}}
			\put(0,-15){\HBCenter{\scriptsize $w_1^{f}$}}
			\put(64,0){\HBCenter{\scriptsize $w_3^{2b}$}}
			\put(30,0){\HBCenter{\scriptsize $v_1^{b+f}$}}
			\put(132,30){\HBCenter{\scriptsize $u_2^{b-f}$}}
			\put(132,-15){\HBCenter{\scriptsize $u_1^{b-f}$}}
			\put(64,-30){\HBCenter{\scriptsize $u_3^{b}$}}
			\put(100,0){\HBCenter{\scriptsize$v_2^{2b-f}$}}
		}
		\qbezier[60](12,-15)(65,-15)(118,-15)
		\qbezier[60](12,30)(65,30)(118,30)
		\put(0,30){\put(6,-6){\vector(1,-1){16}}}
		\put(0,-15){\put(6,3){\vector(2,1){15}}}
		\put(100,0){\put(6,-3){\vector(2,-1){18}}}
		\put(100,0){\put(8,8){\vector(1,1){16}}}
		\put(64,-24){\vector(0,1){18}}
		\put(40,0){\vector(1,0){13}}
		\put(71,0){\vector(1,0){13}}
	}
}
\end{picture}
\end{center}
Proceeding further this way, we see that $T_{v_1}$ is a Hübner-source of 
$\mu_{u_2}\mu_{w_3}\mu_{u_1}\mu_{v_2}(T)$, again since $b\geq f$. The resulting 
quiver after mutation is shown below on the left. Then $T_{v_2}$ is a 
Hübner-sink of $\mu_{v_1}\mu_{u_2}\mu_{w_3}\mu_{u_1}\mu_{v_2}(T)$ since 
$2f-b-2f=-b<0$ since $b=0$ is not possible ($b=0$ would imply $f=0$ and then 
$\frac{a}{b}=\frac{e}{f}=\infty$).
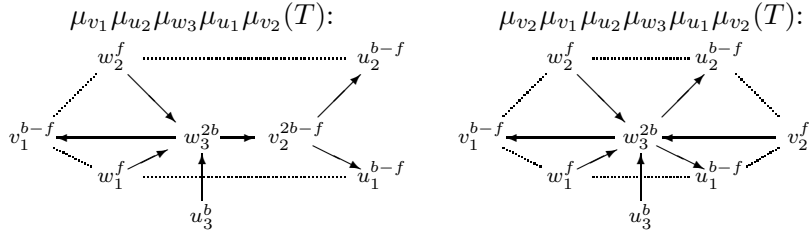
\begin{center}
\begin{picture}(280,95)
\put(0,0){
	\put(0,30){
		\put(64,43){\HBCenter{$\mu_{v_1}\mu_{u_2}\mu_{w_3}\mu_{u_1}\mu_{v_2}(T)$:}}
		\put(0,-2){
			\put(30,30){\HBCenter{\scriptsize $w_2^{f}$}}
			\put(30,-15){\HBCenter{\scriptsize $w_1^{f}$}}
			\put(64,0){\HBCenter{\scriptsize $w_3^{2b}$}}
			\put(0,0){\HBCenter{\scriptsize $v_1^{b-f}$}}
			\put(132,30){\HBCenter{\scriptsize $u_2^{b-f}$}}
			\put(132,-15){\HBCenter{\scriptsize $u_1^{b-f}$}}
			\put(64,-30){\HBCenter{\scriptsize $u_3^{b}$}}
			\put(100,0){\HBCenter{\scriptsize$v_2^{2b-f}$}}
		}
		\qbezier[16](8,8)(16,16)(24,24)
		\qbezier[12](8,-4)(15,-7.5)(22,-11)
		\qbezier[50](42,30)(80,30)(118,30)
		\qbezier[50](42,-15)(80,-15)(118,-15)
		\put(30,30){\put(6,-6){\vector(1,-1){18}}}
		\put(30,-15){\put(6,3){\vector(2,1){15}}}
		\put(100,0){\put(6,-3){\vector(2,-1){18}}}
		\put(100,0){\put(8,8){\vector(1,1){16}}}
		\put(64,-24){\vector(0,1){18}}
		\put(54,0){\vector(-1,0){45}}
		\put(71,0){\vector(1,0){13}}
	}
	\put(170,30){
		\put(64,43){\HBCenter{$\mu_{v_2}\mu_{v_1}\mu_{u_2}\mu_{w_3}\mu_{u_1}\mu_{v_2}(T)$:}}
		\put(0,-2){
			\put(30,30){\HBCenter{\scriptsize $w_2^{f}$}}
			\put(30,-15){\HBCenter{\scriptsize $w_1^{f}$}}
			\put(60,0){\HBCenter{\scriptsize $w_3^{2b}$}}
			\put(0,0){\HBCenter{\scriptsize $v_1^{b-f}$}}
			\put(90,30){\HBCenter{\scriptsize $u_2^{b-f}$}}
			\put(90,-15){\HBCenter{\scriptsize $u_1^{b-f}$}}
			\put(60,-30){\HBCenter{\scriptsize $u_3^{b}$}}
			\put(120,0){\HBCenter{\scriptsize$v_2^{f}$}}
		}
		\qbezier[16](8,8)(16,16)(24,24)
		\qbezier[12](8,-4)(15,-7.5)(22,-11)
		\qbezier[26](42,30)(60,30)(78,30)
		\qbezier[26](42,-15)(60,-15)(78,-15)
		\qbezier[16](112,8)(104,16)(96,24)
		\qbezier[12](112,-4)(106,-7.5)(100,-11)
		\put(30,30){\put(6,-6){\vector(1,-1){18}}}
		\put(30,-15){\put(6,3){\vector(2,1){15}}}
		\put(60,0){\put(6,-3){\vector(2,-1){18}}}
		\put(60,0){\put(8,8){\vector(1,1){16}}}
		\put(60,-24){\vector(0,1){18}}
		\put(50,0){\vector(-1,0){41}}
		\put(111,0){\vector(-1,0){43}}
	}
}
\end{picture}
\end{center}
The final step just changes the direction of the arrow $u_3\rightar w_3$ but 
not the slope. We therefore see that in case (a) we have $\mu_v(T)\in\cT$ and 
$\slopeset(\mu_v(T))=\mu_{\frac{c}{d}}(\slopeset(T))$.
 
We now consider case (b), where we can copy the first step and start with 
$\mu_{v_2}(T)$ already calculated as in case (a) since again $T_{v_2}$ is a 
Hübner-source of $T$. 
Now $T_{u_1}$ is a Hübner-sink of $\mu_{v_2}(T)$.
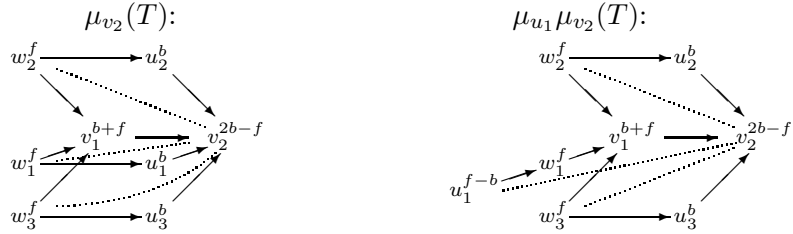
\begin{center}
\begin{picture}(280,90)
\put(0,0){
	\put(0,30){
		\put(40,43){\HBCenter{$\mu_{v_2}(T)$:}}
		\put(0,-2){
			\put(0,30){\HBCenter{\scriptsize $w_2^{f}$}}
			\put(0,-10){\HBCenter{\scriptsize $w_1^{f}$}}
			\put(0,-30){\HBCenter{\scriptsize $w_3^{f}$}}
			\put(30,0){\HBCenter{\scriptsize $v_1^{b+f}$}}
			\put(50,30){\HBCenter{\scriptsize $u_2^{b}$}}
			\put(50,-10){\HBCenter{\scriptsize $u_1^{b}$}}
			\put(50,-30){\HBCenter{\scriptsize $u_3^{b}$}}
			\put(80,0){\HBCenter{\scriptsize$v_2^{2b-f}$}}
		}
		\qbezier[30](12,26)(40,15)(68,4)
		\qbezier[33](12,-26)(50,-25)(72,-6)
		\qbezier[27](12,-8.8)(40,-5)(62,-2)
		\multiput(0,30)(50,0){2}{\put(6,-6){\vector(1,-1){16}}}
		\multiput(0,-30)(50,0){2}{\put(6,6){\vector(1,1){18}}}
		\multiput(0,-10)(50,0){2}{\put(6,2){\vector(3,1){13}}}
		\put(6,30){\vector(1,0){38}}
		\put(6,-10){\vector(1,0){38}}
		\put(6,-30){\vector(1,0){38}}
		\put(42,0){\vector(1,0){20}}
	}
	\put(200,30){
		\put(10,43){\HBCenter{$\mu_{u_1}\mu_{v_2}(T)$:}}
		\put(0,-2){
			\put(0,30){\HBCenter{\scriptsize $w_2^{f}$}}
			\put(0,-10){\HBCenter{\scriptsize $w_1^{f}$}}
			\put(0,-30){\HBCenter{\scriptsize $w_3^{f}$}}
			\put(30,0){\HBCenter{\scriptsize $v_1^{b+f}$}}
			\put(50,30){\HBCenter{\scriptsize $u_2^{b}$}}
			\put(-30,-20){\HBCenter{\scriptsize $u_1^{f-b}$}}
			\put(50,-30){\HBCenter{\scriptsize $u_3^{b}$}}
			\put(80,0){\HBCenter{\scriptsize$v_2^{2b-f}$}}
		}
		\qbezier[30](12,26)(40,15)(68,4)
		\qbezier[30](12,-26)(40,-15)(68,-4)
		\put(-30,-20){\qbezier[50](11,0)(55,9)(99,18)}
		\multiput(0,30)(50,0){2}{\put(6,-6){\vector(1,-1){16}}}
		\multiput(0,-30)(50,0){2}{\put(6,6){\vector(1,1){18}}}
		\multiput(0,-10)(-25.5,-8.5){2}{\put(6,2){\vector(3,1){13}}}
		\put(6,30){\vector(1,0){38}}
		\put(6,-30){\vector(1,0){38}}
		\put(42,0){\vector(1,0){20}}
	}
}
\end{picture}
\end{center}
Clearly $T_{w_3}$ is a Hübner-source of $\mu_{u_1}\mu_{v_2}(T)$ and then
$T_{u_2}$ is a Hübner-sink of $\mu_{w_3}\mu_{u_1}\mu_{v_2}(T)$ since 
$b<f$.
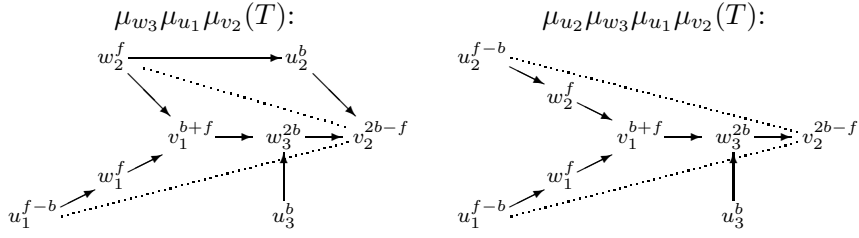
\begin{center}
\begin{picture}(280,90)
\put(0,0){
	\put(30,30){
		\put(35,43){\HBCenter{$\mu_{w_3}\mu_{u_1}\mu_{v_2}(T)$:}}
		\put(0,-2){
			\put(0,30){\HBCenter{\scriptsize $w_2^{f}$}}
			\put(0,-15){\HBCenter{\scriptsize $w_1^{f}$}}
			\put(65,0){\HBCenter{\scriptsize $w_3^{2b}$}}
			\put(30,0){\HBCenter{\scriptsize $v_1^{b+f}$}}
			\put(70,30){\HBCenter{\scriptsize $u_2^{b}$}}
			\put(-30,-30){\HBCenter{\scriptsize $u_1^{f-b}$}}
			\put(65,-30){\HBCenter{\scriptsize $u_3^{b}$}}
			\put(102,0){\HBCenter{\scriptsize$v_2^{2b-f}$}}
		}
		\qbezier[30](12,26)(50,15)(88,4)
		\put(-30,-30){\qbezier[50](11,0)(65,14)(119,28)}
		\multiput(0,30)(70,0){2}{\put(6,-6){\vector(1,-1){16}}}
		\multiput(0,-15)(-25.5,-13.5){2}{\put(6,2){\vector(2,1){13}}}
		\put(6,30){\vector(1,0){58}}
		\put(65,-24){\vector(0,1){18}}
		\multiput(39,0)(34,0){2}{\vector(1,0){14}}
	}
	\put(200,30){
		\put(35,43){\HBCenter{$\mu_{u_2}\mu_{w_3}\mu_{u_1}\mu_{v_2}(T)$:}}
		\put(0,-2){
			\put(0,15){\HBCenter{\scriptsize $w_2^{f}$}}
			\put(0,-15){\HBCenter{\scriptsize $w_1^{f}$}}
			\put(65,0){\HBCenter{\scriptsize $w_3^{2b}$}}
			\put(30,0){\HBCenter{\scriptsize $v_1^{b+f}$}}
			\put(-30,30){\HBCenter{\scriptsize $u_2^{f-b}$}}
			\put(-30,-30){\HBCenter{\scriptsize $u_1^{f-b}$}}
			\put(65,-30){\HBCenter{\scriptsize $u_3^{b}$}}
			\put(102,0){\HBCenter{\scriptsize$v_2^{2b-f}$}}
		}
		\put(-30,-30){\qbezier[50](11,0)(65,14)(119,28)}
		\put(-30,30){\qbezier[50](11,0)(65,-14)(119,-28)}
		\multiput(0,-15)(-25.5,-13.5){2}{\put(6,2){\vector(2,1){13}}}
		\multiput(0,15)(-25.5,13.5){2}{\put(6,-2){\vector(2,-1){13}}}
		\put(65,-24){\vector(0,1){18}}
		\multiput(39,0)(34,0){2}{\vector(1,0){14}}
	}
}
\end{picture}
\end{center}
Now, since $b-f<0$ we have that $T_{v_1}$ is a Hübner-sink of 
$\mu_{u_2}\mu_{w_3}\mu_{u_1}\mu_{v_2}(T)$. The resulting quiver is shown 
on the left hand side in the next picture. Clearly $T_{v_2}$ is a Hübner-sink in
$\mu_{u_2}\mu_{w_3}\mu_{u_1}\mu_{v_2}(T)$.
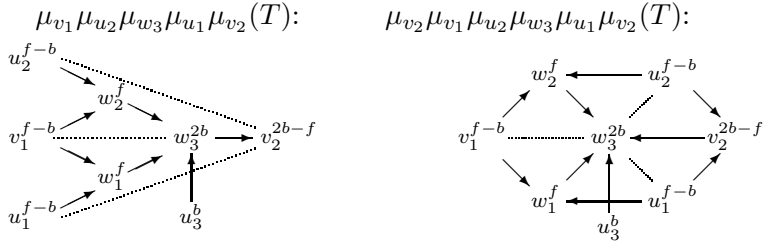
\begin{center}
\begin{picture}(280,90)
\put(0,0){
	\put(30,30){
		\put(21,43){\HBCenter{$\mu_{v_1}\mu_{u_2}\mu_{w_3}\mu_{u_1}\mu_{v_2}(T)$:}}
		\put(0,-2){
			\put(0,15){\HBCenter{\scriptsize $w_2^{f}$}}
			\put(0,-15){\HBCenter{\scriptsize $w_1^{f}$}}
			\put(30,0){\HBCenter{\scriptsize $w_3^{2b}$}}
			\put(-30,0){\HBCenter{\scriptsize $v_1^{f-b}$}}
			\put(-30,30){\HBCenter{\scriptsize $u_2^{f-b}$}}
			\put(-30,-30){\HBCenter{\scriptsize $u_1^{f-b}$}}
			\put(30,-30){\HBCenter{\scriptsize $u_3^{b}$}}
			\put(67,0){\HBCenter{\scriptsize$v_2^{2b-f}$}}
		}
		\put(-30,-30){\qbezier[50](11,0)(48,13)(84,26)}
		\put(-30,30){\qbezier[50](11,0)(48,-13)(84,-26)}
		\qbezier[26](20,0)(-1,0)(-22,0)
		\multiput(0,-15)(-25.5,-13.5){2}{\put(6,2){\vector(2,1){13}}}
		\multiput(0,15)(-25.5,13.5){2}{\put(6,-2){\vector(2,-1){13}}}
		\put(-25.5,0){
			\put(6,-3.5){\vector(2,-1){13}}
			\put(6,3.5){\vector(2,1){13}}
			}
		\put(30,-24){\vector(0,1){18}}
		\put(39,0){\vector(1,0){14}}
	}
	\put(170,30){
		\put(21,43){\HBCenter{$\mu_{v_2}\mu_{v_1}\mu_{u_2}\mu_{w_3}\mu_{u_1}\mu_{v_2}(T)$:}}
		\put(0,-2){
			\put(24,24){\HBCenter{\scriptsize $w_2^{f}$}}
			\put(24,-24){\HBCenter{\scriptsize $w_1^{f}$}}
			\put(48,0){\HBCenter{\scriptsize $w_3^{2b}$}}
			\put(0,0){\HBCenter{\scriptsize $v_1^{f-b}$}}
			\put(72,24){\HBCenter{\scriptsize $u_2^{f-b}$}}
			\put(72,-24){\HBCenter{\scriptsize $u_1^{f-b}$}}
			\put(48,-34){\HBCenter{\scriptsize $u_3^{b}$}}
			\put(96,0){\HBCenter{\scriptsize$v_2^{2b-f}$}}
		}
		\multiput(32,-16)(48,0){2}{\vector(1,1){10}}
		\multiput(32,18)(48,0){2}{\vector(1,-1){10}}
		\multiput(60,-24)(0,48){2}{\vector(-1,0){28}}
		\put(84,0){\vector(-1,0){28}}
		\put(8,8){\vector(1,1){10}}
		\put(8,-8){\vector(1,-1){10}}
		\put(48,-28){\vector(0,1){22}}
		\qbezier[22](10,0)(24,0)(38,0)
		\put(48,0){
			\qbezier[11](8,8)(12,12)(16,16)
			\qbezier[11](8,-8)(12,-12)(16,-16)
		}
	}
}
\end{picture}
\end{center} 
Again, the last step only changes the direction of the arrow $u_3\rightar w_3$ 
but no slope. Hence we have shown that also in case (b) we have 
$\mu_v(T)\in\cT$ and that $\slopeset(\mu_v(T))=\mu_{\frac{c}{d}}(\slopeset(T))$.

We consider now the case (c) where $T_{v_2}$ is a Hübner-sink of $T$.
The resulting situation is shown in the next picture on the left.
Clearly $T_{u_1}$ is a sink of $\mu_{v_2}(T)$.
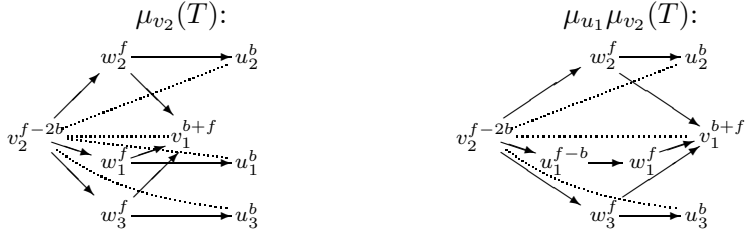
\begin{center}
\begin{picture}(280,90)
\put(0,0){
	\put(30,30){
		\put(25,43){\HBCenter{$\mu_{v_2}(T)$:}}
		\put(0,-2){
			\put(0,30){\HBCenter{\scriptsize $w_2^{f}$}}
			\put(0,-10){\HBCenter{\scriptsize $w_1^{f}$}}
			\put(0,-30){\HBCenter{\scriptsize $w_3^{f}$}}
			\put(30,0){\HBCenter{\scriptsize $v_1^{b+f}$}}
			\put(50,30){\HBCenter{\scriptsize $u_2^{b}$}}
			\put(50,-10){\HBCenter{\scriptsize $u_1^{b}$}}
			\put(50,-30){\HBCenter{\scriptsize $u_3^{b}$}}
			\put(-30,-1){\HBCenter{\scriptsize$v_2^{f-2b}$}}
		}
		\qbezier[40](-20,3)(11,15)(42,27)
		\qbezier[40](-22,-5)(1,-23)(42,-27)
		\qbezier[35](-18,-1)(12,-4.5)(42,-8)
		\qbezier[22](-18,0)(1,0)(20,0)
		\multiput(0,30)(-30,-30){2}{\put(6,-6){\vector(1,-1){16}}}
		\multiput(0,-30)(-30,30){2}{\put(7,7){\vector(1,1){17}}}
		\put(0,-10){\put(6,2){\vector(3,1){13}}}
		\put(-30,0){\put(6,-2){\vector(3,-1){15}}}
		\put(6,30){\vector(1,0){38}}
		\put(6,-10){\vector(1,0){38}}
		\put(6,-30){\vector(1,0){38}}
	}
	\put(200,30){
		\put(25,43){\HBCenter{$\mu_{u_1}\mu_{v_2}(T)$:}}
		\put(0,-2){
			\put(15,30){\HBCenter{\scriptsize $w_2^{f}$}}
			\put(30,-10){\HBCenter{\scriptsize $w_1^{f}$}}
			\put(15,-30){\HBCenter{\scriptsize $w_3^{f}$}}
			\put(60,0){\HBCenter{\scriptsize $v_1^{b+f}$}}
			\put(50,30){\HBCenter{\scriptsize $u_2^{b}$}}
			\put(0,-10){\HBCenter{\scriptsize $u_1^{f-b}$}}
			\put(50,-30){\HBCenter{\scriptsize $u_3^{b}$}}
			\put(-30,-1){\HBCenter{\scriptsize$v_2^{f-2b}$}}
		}
		\qbezier[40](-20,3)(11,15)(42,27)
		\qbezier[40](-22,-5)(1,-23)(42,-27)
		\qbezier[32](-18,0)(14,0)(46,0)
		\multiput(-30,0)(45,30){2}{\put(6,-6){\vector(3,-2){30}}}
		\multiput(-30,0)(45,-30){2}{\put(7,7){\vector(3,2){29}}}
		\put(30,-10){\put(6,4){\vector(3,1){13}}}
		\put(-30,0){\put(6,-2){\vector(3,-1){13}}}
		\put(21,30){\vector(1,0){23}}
		\put(9,-10){\vector(1,0){13}}
		\put(21,-30){\vector(1,0){23}}
	}
}
\end{picture}
\end{center}
Since $b\geq 0$ we have that $T_{w_3}$ is Hübner-source of 
$\mu_{u_1}\mu_{v_2}(T)$. Also $T_{u_2}$ is a Hübner-sink
of the resulting $\mu_{w_3}\mu_{u_1}\mu_{v_2}(T)$.
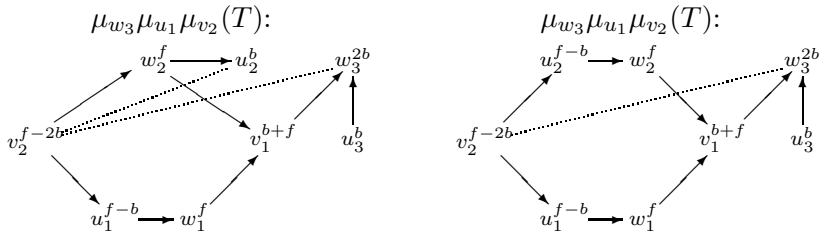
\begin{center}
\begin{picture}(280,90)
\put(0,0){
	\put(30,30){
		\put(25,43){\HBCenter{$\mu_{w_3}\mu_{u_1}\mu_{v_2}(T)$:}}
		\put(0,-2){
			\put(15,30){\HBCenter{\scriptsize $w_2^{f}$}}
			\put(30,-30){\HBCenter{\scriptsize $w_1^{f}$}}
			\put(90,30){\HBCenter{\scriptsize $w_3^{2b}$}}
			\put(60,0){\HBCenter{\scriptsize $v_1^{b+f}$}}
			\put(50,30){\HBCenter{\scriptsize $u_2^{b}$}}
			\put(0,-30){\HBCenter{\scriptsize $u_1^{f-b}$}}
			\put(90,0){\HBCenter{\scriptsize $u_3^{b}$}}
			\put(-30,-1){\HBCenter{\scriptsize$v_2^{f-2b}$}}
		}
		\qbezier[40](-20,3)(11,15)(42,27)
		\qbezier[60](-20,2)(31,15)(82,27)
		\put(90,6){\vector(0,1){18}}
		\put(15,30){\put(6,-6){\vector(3,-2){30}}}
		\put(-30,0){\put(7,7){\vector(3,2){29}}}
		\put(30,-30){\put(6,6){\vector(1,1){18}}}
		\put(62,2){\put(6,6){\vector(1,1){18}}}
		\put(-30,0){\put(6,-6){\vector(1,-1){18}}}
		\put(21,30){\vector(1,0){23}}
		\put(9,-30){\vector(1,0){13}}
	}
	\put(200,30){
		\put(25,43){\HBCenter{$\mu_{w_3}\mu_{u_1}\mu_{v_2}(T)$:}}
		\put(0,-2){
			\put(30,30){\HBCenter{\scriptsize $w_2^{f}$}}
			\put(30,-30){\HBCenter{\scriptsize $w_1^{f}$}}
			\put(90,30){\HBCenter{\scriptsize $w_3^{2b}$}}
			\put(60,0){\HBCenter{\scriptsize $v_1^{b+f}$}}
			\put(0,30){\HBCenter{\scriptsize $u_2^{f-b}$}}
			\put(0,-30){\HBCenter{\scriptsize $u_1^{f-b}$}}
			\put(90,0){\HBCenter{\scriptsize $u_3^{b}$}}
			\put(-30,-1){\HBCenter{\scriptsize$v_2^{f-2b}$}}
		}
		\qbezier[60](-20,2)(31,15)(82,27)
		\put(90,6){\vector(0,1){18}}
		\put(30,30){\put(6,-6){\vector(1,-1){18}}}
		\put(-30,0){\put(7,7){\vector(1,1){18}}}
		\put(30,-30){\put(6,6){\vector(1,1){18}}}
		\put(62,2){\put(6,6){\vector(1,1){18}}}
		\put(-30,0){\put(6,-6){\vector(1,-1){18}}}
		\put(9,30){\vector(1,0){13}}
		\put(9,-30){\vector(1,0){13}}
	}
}
\end{picture}
\end{center}
Again, since $b-f<0$ we have that $T_{v_1}$ is a Hübner-sink of 
$\mu_{w_3}\mu_{u_1}\mu_{v_2}(T)$. Clearly $T_{v_2}$ is a Hübner-source of 
$\mu_{v_1}\mu_{u_2}\mu_{w_3}\mu_{u_1}\mu_{v_2}(T)$.

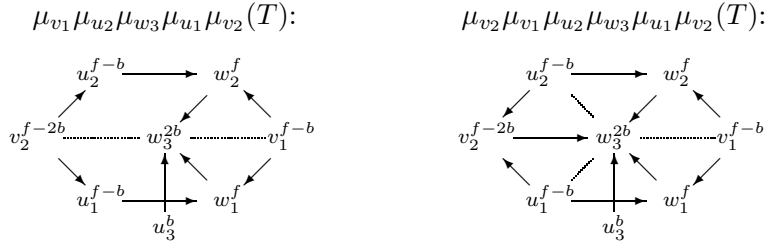
\begin{center}
\begin{picture}(280,90)
\put(0,0){
	\put(0,30){
		\put(48,43){\HBCenter{$\mu_{v_1}\mu_{u_2}\mu_{w_3}\mu_{u_1}\mu_{v_2}(T)$:}}
		\put(0,-2){
			\put(72,24){\HBCenter{\scriptsize $w_2^{f}$}}
			\put(72,-24){\HBCenter{\scriptsize $w_1^{f}$}}
			\put(48,0){\HBCenter{\scriptsize $w_3^{2b}$}}
			\put(96,0){\HBCenter{\scriptsize $v_1^{f-b}$}}
			\put(24,24){\HBCenter{\scriptsize $u_2^{f-b}$}}
			\put(24,-24){\HBCenter{\scriptsize $u_1^{f-b}$}}
			\put(48,-34){\HBCenter{\scriptsize $u_3^{b}$}}
			\put(0,0){\HBCenter{\scriptsize$v_2^{f-2b}$}}
		}
		\put(88,-8){\vector(-1,-1){10}}
		\put(88,8){\vector(-1,1){10}}
		\put(64,-16){\vector(-1,1){10}}
		\put(64,16){\vector(-1,-1){10}}
		\multiput(32,-24)(0,48){2}{\vector(1,0){28}}
		\put(8,8){\vector(1,1){10}}
		\put(8,-8){\vector(1,-1){10}}
		\put(48,-28){\vector(0,1){22}}
		\qbezier[22](10,0)(24,0)(38,0)
		\put(48,0){\qbezier[22](10,0)(24,0)(38,0)}
	}
	\put(170,30){
		\put(48,43){\HBCenter{$\mu_{v_2}\mu_{v_1}\mu_{u_2}\mu_{w_3}\mu_{u_1}\mu_{v_2}(T)$:}}
		\put(0,-2){
			\put(72,24){\HBCenter{\scriptsize $w_2^{f}$}}
			\put(72,-24){\HBCenter{\scriptsize $w_1^{f}$}}
			\put(48,0){\HBCenter{\scriptsize $w_3^{2b}$}}
			\put(96,0){\HBCenter{\scriptsize $v_1^{f-b}$}}
			\put(24,24){\HBCenter{\scriptsize $u_2^{f-b}$}}
			\put(24,-24){\HBCenter{\scriptsize $u_1^{f-b}$}}
			\put(48,-34){\HBCenter{\scriptsize $u_3^{b}$}}
			\put(0,0){\HBCenter{\scriptsize$v_2^{f-2b}$}}
		}
		\put(88,-8){\vector(-1,-1){10}}
		\put(88,8){\vector(-1,1){10}}
		\put(64,-16){\vector(-1,1){10}}
		\put(64,16){\vector(-1,-1){10}}
		\multiput(32,-24)(0,48){2}{\vector(1,0){28}}
		\put(16,18){\vector(-1,-1){10}}
		\put(16,-18){\vector(-1,1){10}}
		\put(48,-28){\vector(0,1){22}}
		\put(10,0){\vector(1,0){28}}
		\put(48,0){
			\qbezier[11](-8,8)(-12,12)(-16,16)
			\qbezier[11](-8,-8)(-12,-12)(-16,-16)
		}
		\put(48,0){\qbezier[22](10,0)(24,0)(38,0)}
	}
}
\end{picture}
\end{center}
And again the last mutation in the sequence does only revert the direction of 
the arrow
$u_3\rightar w_3$ without changing any slope. Hence we have proved the 
assertion in case (c) and therefore in all possible cases where the quiver of 
$\End(T)$ is $Q_w$.

A similar calculation deals with the cases where $\End(T)$ has as quiver $Q_u$ 
or  $Q_v$. However these cases need distinction since in no step there is a 
possible choice for a vertex to be a Hübner-source or a Hübner-sink. These 
cases are therefore left to the interested reader. This concludes the proof of 
the statement.
\end{proof}

Since the remaining two tubular weight sequences require no new type of 
argument we restrict to give the proper definition and statement and leave the 
verification to the interested reader.

If the weight sequence is $(4,4,2)$, 
let  $\cT_{(4,4,2)}$ be the set of (isomorphism classes of) tilting sheaves
\[
T=\oplus_{i\in I} T_i \text{ with } I=\{u_1,u_2,u_3,v_1,v_2,v_3,w_1,w_2,w_3\},
\]
such that  
\begin{itemize}
\item the rank and degree of the indecomposable summands of $T$ are 
related as shown in the following table:
\vspace{.5ex}
\begin{center}
\begin{picture}(210,40)
\multiput(0,0)(0,20){3}{\line(1,0){210}}
\put(0,0){\line(0,1){40}}
\multiput(30,0)(20,0){10}{\line(0,1){40}}
\put(15,28){\HBCenter{$i$}}
\put(15,8){\HBCenter{$\frac{\deg(T_i)}{\rk(T_i)}$}}
\put(40,28){\HBCenter{$u_1$}}
\put(60,28){\HBCenter{$u_2$}}
\put(80,28){\HBCenter{$u_3$}}
\multiput(40,8)(40,0){2}{\HBCenter{$\frac{a}{b}$}}
\put(60,8){\HBCenter{$\frac{2a}{2b}$}}
\put(100,28){\HBCenter{$v_1$}}
\put(120,28){\HBCenter{$v_2$}}
\put(140,28){\HBCenter{$v_3$}}
\multiput(100,8)(40,0){2}{\HBCenter{$\frac{c}{d}$}}
\put(120,8){\HBCenter{$\frac{2c}{2d}$}}
\put(160,28){\HBCenter{$w_1$}}
\put(180,28){\HBCenter{$w_2$}}
\put(200,28){\HBCenter{$w_3$}}
\multiput(160,8)(40,0){2}{\HBCenter{$\frac{e}{f}$}}
\put(180,8){\HBCenter{$\frac{2e}{2f}$}}
\end{picture}
\end{center}
\item
$\slopeset(T)=\{\frac{a}{b},\frac{c}{d}, \frac{e}{f}\}$ is a Farey triple,
\item
the quiver with relations of $\End(T)$ is one shown in the middle of 
Figure~\ref{fig:ex-real}.
\end{itemize}
Note that $\cT_{4,4,2}$ is not empty by the construction in 
Section~\ref{ssec:TiFa}.

The mutation sequences in this case are then defined to be
\begin{align*}
\mu_u&=\mu_{w_1}\mu_{u_3}\mu_{v_3}\mu_{v_1}\mu_{u_2}\mu_{v_2}\mu_{u_2}\mu_{w_2}\mu_{u_2}\mu_{v_3},\\
\mu_v&=\mu_{u_1}\mu_{v_3}\mu_{w_3}\mu_{w_1}\mu_{v_2}\mu_{w_2}\mu_{v_2}\mu_{u_2}\mu_{v_2}\mu_{w_3},\\
\mu_w&=\mu_{v_1}\mu_{w_3}\mu_{u_3}\mu_{u_1}\mu_{c_2}\mu_{u_2}\mu_{w_2}\mu_{v_2}\mu_{w_2}\mu_{u_3}.
\end{align*}

\begin{prp} \label{prp442}
Let $\XX$ be a weighted projective line with weight sequence $(4,4,2)$.
Then for any $T\in\cT_{(4,4,2)}$ and $x\in{u,v,w}$ we have 
$\mu_x(T)\in \cT_{(4,4,2)}$ and $\slopeset(\mu_x(T))=\mu_{q_x}(\slopeset(T))$, 
where $q_x=\slope(T_{x_1})$.
\end{prp}

In case the weight sequence is $(6,3,2)$ we need to
define three quivers $Q_x$ for $x\in\{u,v,w\}$ as follows:

\begin{figure}[!ht]
\begin{picture}(340,100)
\put(50,40){
	\put(-57,50){$Q_u$:}
	\put(0,0){
		\put(0,-20){\HVCenter{\small $v_3$}}
		\put(0,20){\HVCenter{\small $u_3$}}
		\put(-20,0){\HVCenter{\small $w_2$}}
		\put(-20,-40){\HVCenter{\small $u_1$}}
		\put(-20,40){\HVCenter{\small $v_2$}}
		\put(-44,0){\HVCenter{\small $w_1$}}
		\put(18,0){\HVCenter{\small $w_4$}}
		\put(20,-40){\HVCenter{\small $u_2$}}
		\put(20,40){\HVCenter{\small $v_1$}}
		\put(44,0){\HVCenter{\small $w_3$}}
	}
	\multiput(-20,40)(40,0){2}{\put(0,-8){\vector(0,-1){24}}}
	\put(0,20){\put(0,-8){\vector(0,-1){24}}}
	\multiput(0,-20)(0,40){2}{
		\put(5,5){\vector(1,1){10}}
	}
	\multiput(0,-20)(0,40){2}{
		\put(-5,5){\vector(-1,1){10}}
	}
	\put(20,-40){\put(-5,5){\vector(-1,1){10}}}
	\put(-20,-40){\put(5,5){\vector(1,1){10}}}
	\put(-26,0){\vector(-1,0){11}}
	\put(26,0){\vector(1,0){11}}
	\put(-20,0){\qbezier[12](5,5)(10,10)(15,15)}
	\put(20,0){\qbezier[12](-5,5)(-10,10)(-15,15)}
	\multiput(-20,-20)(40,0){2}{\qbezier[20](0,12)(0,0)(0,-12)}
}
\put(165,40){
	\put(-57,50){$Q_v$:}
	\put(0,0){
		\put(0,-20){\HVCenter{\small $v_3$}}
		\put(0,20){\HVCenter{\small $u_3$}}
		\put(-20,0){\HVCenter{\small $w_2$}}
		\put(-20,-40){\HVCenter{\small $u_1$}}
		\put(-20,40){\HVCenter{\small $v_2$}}
		\put(-44,0){\HVCenter{\small $w_1$}}
		\put(18,0){\HVCenter{\small $w_4$}}
		\put(20,-40){\HVCenter{\small $u_2$}}
		\put(20,40){\HVCenter{\small $v_1$}}
		\put(44,0){\HVCenter{\small $w_3$}}
	}
	\multiput(-20,40)(40,0){2}{\put(0,-8){\vector(0,-1){24}}}
	\multiput(-20,0)(40,0){2}{\put(0,-8){\vector(0,-1){24}}}
	\multiput(0,-20)(-20,20){2}{
		\put(5,5){\vector(1,1){10}}
	}
	\multiput(0,-20)(20,20){2}{
		\put(-5,5){\vector(-1,1){10}}
	}
	\qbezier[20](0,12)(0,0)(0,-12)
	\put(-26,0){\vector(-1,0){11}}
	\put(26,0){\vector(1,0){11}}
	\multiput(0,20)(-20,-60){2}{\qbezier[12](5,5)(10,10)(15,15)}
	\multiput(0,20)(20,-60){2}{\qbezier[12](-5,5)(-10,10)(-15,15)}
}
\put(280,40){
	\put(-57,50){$Q_w$:}
	\put(0,0){
		\put(0,-20){\HVCenter{\small $v_3$}}
		\put(0,20){\HVCenter{\small $u_3$}}
		\put(-20,0){\HVCenter{\small $w_2$}}
		\put(-20,-40){\HVCenter{\small $u_1$}}
		\put(-20,40){\HVCenter{\small $v_2$}}
		\put(-44,0){\HVCenter{\small $w_1$}}
		\put(18,0){\HVCenter{\small $w_4$}}
		\put(20,-40){\HVCenter{\small $u_2$}}
		\put(20,40){\HVCenter{\small $v_1$}}
		\put(44,0){\HVCenter{\small $w_3$}}
	}
	\multiput(-20,0)(40,0){2}{\put(0,-8){\vector(0,-1){24}}}
	\put(0,20){\put(0,-8){\vector(0,-1){24}}}
	\multiput(-20,-40)(20,60){2}{
		\put(5,5){\vector(1,1){10}}
	}
	\multiput(20,-40)(-20,60){2}{
		\put(-5,5){\vector(-1,1){10}}
	}
	\put(20,0){\put(-5,5){\vector(-1,1){10}}}
	\put(-20,0){\put(5,5){\vector(1,1){10}}}
	\put(-26,0){\vector(-1,0){11}}
	\put(26,0){\vector(1,0){11}}
	\put(0,-20){\qbezier[12](5,5)(10,10)(15,15)}
	\put(0,-20){\qbezier[12](-5,5)(-10,10)(-15,15)}
	\multiput(-20,20)(40,0){2}{\qbezier[20](0,12)(0,0)(0,-12)}
}
\end{picture}
\caption{Quivers with relations for the weight sequence $(6,3,2)$.}
\label{fig:quivers-632}
\end{figure}
Let $\cT_{(6,3,2)}$ to be the set of (isomorphism classes of) all tilting sheaves
\[
T=\oplus_{i\in I} T_i \text{ with } I=\{u_1,u_2,u_3,v_1,v_2,v_3,v_4,w_1,w_2,w_3\},
\]
such that
\begin{itemize}
\item the  rank and degree of the indecomposable summands of $T$ are 
related as shown in the following table:
\vspace{.5ex}
\begin{center}
\begin{picture}(230,40)
\multiput(0,0)(0,20){3}{\line(1,0){230}}
\put(0,0){\line(0,1){40}}
\multiput(30,0)(20,0){11}{\line(0,1){40}}
\put(15,28){\HBCenter{$i$}}
\put(15,8){\HBCenter{$\frac{\deg(T_i)}{\rk(T_i)}$}}
\put(40,28){\HBCenter{$u_1$}}
\put(60,28){\HBCenter{$u_2$}}
\put(80,28){\HBCenter{$u_3$}}
\multiput(40,8)(20,0){2}{\HBCenter{$\frac{a}{b}$}}
\put(80,8){\HBCenter{$\frac{2a}{2b}$}}
\put(100,28){\HBCenter{$v_1$}}
\put(120,28){\HBCenter{$v_2$}}
\put(140,28){\HBCenter{$v_3$}}
\put(160,28){\HBCenter{$v_4$}}
\multiput(100,8)(40,0){2}{\HBCenter{$\frac{c}{d}$}}
\multiput(120,8)(40,0){2}{\HBCenter{$\frac{2c}{2d}$}}
\put(180,28){\HBCenter{$w_1$}}
\put(200,28){\HBCenter{$w_2$}}
\put(220,28){\HBCenter{$w_3$}}
\multiput(180,8)(20,0){2}{\HBCenter{$\frac{e}{f}$}}
\put(220,8){\HBCenter{$\frac{2e}{2f}$}}
\end{picture}
\end{center}
\item
$\slopeset(T)=\{\frac{a}{b},\frac{c}{d}, \frac{e}{f}\}$ is a Farey triple,
\item
the quiver with relations of $\End(T)$ is one of the three quivers 
of Figure~\ref{fig:quivers-632}. 
\end{itemize}
Note that $\cT_{6,3,2}$ is not empty by the construction in 
Section~\ref{ssec:TiFa}.

The mutations sequences are defined as
\begin{align*}
\mu_u&=\mu_{u_3}\mu_{u_2}\mu_{u_1},\\
\mu_v&=\mu_{v_3}\mu_{v_2}\mu_{v_1},\\
\mu_w&=\mu_{w_3}\mu_{w_1}\mu_{v_3}\mu_{w_4}\mu_{u_2}\mu_{w_3}\mu_{w_4}\mu_{w_2}\mu_{u_1}\mu_{w_1}\mu_{w_2}.
\end{align*}

\begin{prp} \label{prp632}
Let $\XX$ be a weighted projective line with weight sequence $(6,3,2)$.
Then for any $T\in\cT_{(6,3,2)}$ and $x\in\{u,v,w\}$ we have 
$\mu_x(T)\in \cT_{(6,3,2)}$ and $\slopeset(\mu_x(T))=\mu_{q_x}(\slopeset(T))$, 
where $q_x=\slope(T_{x_1})$.
\end{prp}

\begin{rem}
Another warning seems in place:
Although the mutation sequences $\mu_u$, $\mu_v$ and $\mu_w$ are 
involutive on Farey triples in the sense of Lemma~\ref{lem:farey_mutation},
they are not
involutive on the isomorphism classes of tilting sheaves in the four tubular 
cases.
\end{rem}

\subsection{Second proof of Theorem~\ref{thm:main}}
For the four tubular types $(2,2,2,2)$, $(3,3,3)$, $(4,4,2)$ and $(6,3,2)$ the 
Propositions~\ref{prp2222},
\ref{prp333}, \ref{prp442} and~\ref{prp632}
respectively provide a recursive procedure to construct a 
$k$-embedding of the 3-regular tree $\TT_3$  (identified with the exchange
graph of Farey-triples) into the exchange graph of 
tilting sheaves over the weighted projective line of the corresponding
tubular weight type. Now, this exchange graph can be identified 
by~\cite{BaKuLe10} with the exchange graph of cluster tilting objects in
the respective tubular cluster category. Thus, in view of 
Section~\ref{ssec:beg-pf} we are done.\hfill$\Box$


\end{document}